\documentclass[11pt, reqno]{amsart}
\usepackage{amssymb}
\usepackage{amsmath}
\usepackage{amsthm}
\usepackage{verbatim}
\usepackage[all]{xy}
\usepackage{url}
\usepackage{mathtools}
\usepackage{dsfont}
\usepackage{pdflscape}
\usepackage[normalem]{ulem}
\usepackage{extarrows}

\usepackage[mathscr]{euscript}
\usepackage{calligra}
\usepackage[T1]{fontenc}
\usepackage{tikz-cd}

\usepackage[top=2.5cm, bottom=2.7cm, left=2.8cm, right=2.8cm]{geometry}

\DeclareFontEncoding{OT2}{}{} 
\DeclareTextFontCommand{\textcyr}{\fontencoding{OT2}
    \fontfamily{wncyr}\fontseries{m}\fontshape{n}\selectfont}
\newcommand{\Sha}{\textcyr{\!Sh}}

\DeclareSymbolFont{rsfs}{U}{rsfs}{m}{n}
\DeclareSymbolFontAlphabet{\mathrsfs}{rsfs}

\theoremstyle{plain}
\newtheorem{theorem}{Theorem}[section]

\newtheorem{proposition}[theorem]{Proposition}
\newtheorem{lemma}[theorem]{Lemma}
\newtheorem{corollary}[theorem]{Corollary}
\newtheorem{question}[theorem]{Question}

\newtheorem{thm}{Theorem}[subsection]
\newtheorem{prop}[thm]{Proposition}%
\newtheorem{lem}[thm]{Lemma}

\theoremstyle{definition}

\newtheorem{example}[theorem]{Example}

\newtheorem{remark}[theorem]{Remark}
\newtheorem{definition}[theorem]{Definition}
\newtheorem{construction}[theorem]{Construction}

\newtheorem{cons}[thm]{Construction}

\newtheorem*{notation*}{Notation}

\DeclareMathOperator{\Aut}{Aut}
\DeclareMathOperator{\Br}{Br}

\DeclareMathOperator{\Char}{char}

\DeclareMathOperator{\coker}{coker}
\DeclareMathOperator{\Cor}{Cor}

\DeclareMathOperator{\Gal}{Gal}

\DeclareMathOperator{\Hom}{Hom}
\DeclareMathOperator{\im}{im}

\DeclareMathOperator{\Res}{Res}

\DeclareMathOperator{\Spec}{Spec}

\DeclareMathOperator{\ind}{ind}
\DeclareMathOperator{\per}{per}
\DeclareMathOperator{\abb}{ab}
\DeclareMathOperator{\locinf}{\loc_\infty}

\newcommand{\C}{{\mathds C}}
\newcommand{\R}{{\mathds R}}
\newcommand{\Q}{{\mathds Q}}

\newcommand{\Z}{{\mathds Z}}

\newcommand{\GG}{{\mathbf G}}

\newcommand{\QQ}{{Q}}

\newcommand{\HH}{{\bf H}}

\newcommand{\CCC}{{\sf C}}
\newcommand{\DDD}{{\sf D}}
\newcommand{\EEE}{{\sf E}}

\newcommand{\bs}{\hs\backslash\hs}
\newcommand{\ov}{\overline}

\newcommand{\lra}{\longrightarrow}

\newcommand{\ii}{{\boldsymbol{i}}}

\newcommand{\ad}{{\rm ad}}
\newcommand{\ab}{{\rm ab}}

\newcommand{\GL}{{\rm GL}}
\newcommand{\SL}{{\rm SL}}

\newcommand{\PGL}{{\rm PGL}}

\newcommand{\ssc}{{\rm sc}}

\newcommand{\into}{\hookrightarrow}
\newcommand{\onto}{\twoheadrightarrow}

\newcommand{\labelt}[1]{\xrightarrow{\makebox[1.2em]{\scriptsize ${#1}$}}}
\newcommand{\labelto}[1]{\xrightarrow{\makebox[1.5em]{\scriptsize ${#1}$}}}

\newcommand{\longisoto}{{\ \labelt{\raisebox{-1.ex}{$\sim$}}\ }}

\newcommand{\isoto}{\longisoto}

\newcommand{\hs}{\kern 0.8pt}
\newcommand{\hssh}{\kern 1.2pt}
\newcommand{\hshs}{\kern 1.6pt}
\newcommand{\hssss}{\kern 2.0pt}

\newcommand{\hm}{\kern -0.8pt}
\newcommand{\hmm}{\kern -1.2pt}

\newcommand{\emm}{\bfseries}

\newcommand{\X}{{{\sf X}}}
\newcommand{\id}{{\rm id}}

\newcommand{\V}{\mathcal{V}}

\newcommand{\sep}{{\rm sep}}

\newcommand{\loc}{{\rm loc}}
\newcommand{\alg}{{\rm alg}}

\newcommand{\G}{\Gamma}

\newcommand{\tors}{{\rm Tors}}

\newcommand{\Gt}{{\G\hm,\hs\tors}}

\newcommand{\vt}{\vartheta}

\newcommand{\Ho}{{\rm H}}
\newcommand{\Hon}{\Ho^1\hm}
\newcommand{\Zl}{{\rm Z}}
\newcommand{\di}{\Diamond}
\newcommand{\dit}{\nabla}

\newcommand{\functor}{{\,\rightsquigarrow\,}}
\newcommand{\Tors}{{\rm Tors}}
\newcommand{\bv}{{\breve{v}}}
\newcommand{\bw}{{\breve{w}}}
\newcommand{\Gvt}{{\G\hm_v\hs,\Tors}}
\newcommand{\Gbvt}{{\G\hm_\bv\hs,\Tors}}
\newcommand{\Gbwt}{{\G\hm_\bw\hs,\Tors}}

\newcommand{\lmod}{\setminus}
\renewcommand{\Gal}{\G}

\newcommand{\cp}{{\rm cp}}

\newcommand{\SO}{{\rm SO}}

\newcommand{\ur}{{\rm ur}}

\newcommand{\vk}{{\varkappa}}

\newcommand{\tame}{{\rm tame}}
\newcommand{\pp}{{\mathfrak p}}

\newcommand{\CC}{{\rm CC}}

\newcommand{\Teich}{{\mathfrak T}}

\newcommand{\ve}{{\varepsilon}}

\newcommand{\Th}{{\Theta}}

\newcommand{\cH}{{\mathcal H}}
\newcommand{\bdot}{{\boldsymbol{\cdot}}}
\newcommand{\ga}{{\gamma}}

\newcommand{\Tsf}{{\sf T\hm}}

\renewcommand{\le}{\leqslant}
\renewcommand{\ge}{\geqslant}

\newcommand{\mg}{{\rm gen}}

\begin{document}

\title[The power operation in Galois cohomology]%
{The power operation in the Galois cohomology\\ of a reductive group over a global field}

\author[Mikhail Borovoi and Zinovy Reichstein {\tiny with appendices by} M. Borovoi and P. Gille]%
{Mikhail Borovoi and Zinovy Reichstein\\
{\tiny with appendices by}\\
Mikhail Borovoi and Philippe Gille}

\address{Raymond and Beverly Sackler School of Mathematical Sciences,
Tel Aviv University, 6997801 Tel Aviv, Israel}
\email{borovoi@tauex.tau.ac.il}

\address{Department of Mathematics, University of British Columbia,
Vancouver, BC V6T 1Z2, Canada}
\email{reichst@math.ubc.ca}

\address{UMR 5208 du CNRS, Institut Camille Jordan, Universit\'e Claude Bernard Lyon 1,
43 boulevard du 11 novembre 1918, 69622 Villeurbanne cedex, France,
and Institute of Mathematics ``Simion Stoilow'' of the Romanian Academy,
21 Calea Grivitei Street, 010702 Bucharest, Romania}
\email{gille@math.univ-lyon1.fr}

\thanks{Mikhail Borovoi was partially supported
by the Israel Science Foundation (grant 1030/22).}

\thanks{Zinovy Reichstein was partially supported by
an Individual Discovery Grant  RGPIN-2023-03353 from the
National Sciences and Engineering Research Council of Canada.}

\thanks{Philippe Gille was  supported by the project ``Group schemes, root systems, and related representations" founded by the European Union - NextGenerationEU through Romania's National Recovery and Resilience Plan (PNRR) call no. PNRR-III-C9-2023-I8, Project CF159/31.07.2023, and coordinated by the Ministry of Research, Innovation and Digitalization (MCID) of Romania.}



\date{\today}

\begin{abstract}
For a connected reductive group $G$ over
a local or global  field $K$, we define
a {\em diamond} (or {\em power}) operation
$$(\xi,n)\mapsto \xi^{\di  n}\,\colon\, \Hon(K,G)\times \Z\to \Hon(K,G)$$
of raising to power $n$ in the Galois cohomology pointed set.
This operation is new when $K$ is a number field.
 We show that this power operation
has many good properties. When $G$ is a torus,
the set $\Hon(K,G)$ has a natural group structure, and $\xi^{\di n}$
then coincides with the $n$-th power of $\xi$ in this group.

On the other hand, we show that a power operation on $\Hon(K,G)$, functorial in $G$,
which we define over local and global fields,
cannot be defined for an arbitrary field $K$.
Our proof of this assertion relies on the results of  Appendix B written by Philippe Gille.

Using the power operation,
for a cohomology class $\xi$ in $\Hon(K,G)$ over local or global field,
 we  define the period  $\per(\xi)$
to be the least integer $n\ge 1$ such that $\xi^{\di n}=1$.
We define the index $\ind(\xi)$ to be the greatest common divisor of the degrees $[L:K]$
of finite extensions $L/K$ splitting $\xi$.
The period and index of a cohomology class generalize the period and index a central simple algebra over $K$.
For any  connected  reductive group $G$ defined over a local or global field $K$,
we show  that  $\per(\xi)$ divides $\ind(\xi)$
and that $\ind(\xi)$ may be strictly greater than
$\per(\xi)$, but they always  have the same prime factors.
\end{abstract}

\date{\today}

\maketitle

\tableofcontents

\section{Introduction}
\label{s:intro}

Let $G$ be a reductive algebraic group defined over a field $K$.
Here and throughout this paper we shall
follow the convention of SGA3 where reductive groups are assumed to be connected.
We are interested in the Galois cohomology set $\Hon(K, G)$.
Elements of this set can be defined in cohomological terms as
equivalence classes of cocycles,
or in geometric terms as isomorphism classes of $G$-torsors over $\Spec(K)$.
There is also an algebraic interpretation, which makes these sets particularly interesting:
if $A$ is a finite-dimensional vector space with a tensor over $K$
or a  quasi-projective variety with additional structure  over $K$,
and $G$ is the automorphism group of $A$ over an algebraic closure $\ov K$ of $K$,
then $\Hon(K, G)$ classifies $K$-forms of $A$,
that is, objects $B$ defined over $K$ such that $B$
becomes isomorphic to $A$ over a separable closure $K^s$ of $K$;
see Serre \cite[Section III.1]{Serre-GC}.
For example,
$\Hon(K, \operatorname{O}_n)$ classifies non-degenerate $n$-dimensional quadratic forms over $K$,
$\Hon(K, {\sf G}_2)$ classifies octonion algebras over $K$, and so on.
This point of view, pioneered by Serre and Springer in the early 1960s,
has been very successful in unifying previously disparate and ad-hoc constructions
in various branches of algebra
and in introducing cohomological machinery into the subject. A complicating feature of this approach,
reflecting intrinsic complexity of the subject, is that
$\Hon(K, G)$ is only a set with a
distinguished element (called the neutral element).
In general it does not admit a functorial group structure; see \cite{Borovoi-NG}.

In this paper  we attempt to define a power operation
$(\xi,n)\mapsto \xi^{\di n}$ on $\Hon(K, G)$.
A power operation is weaker than a group operation,
but if it can be defined, it will give us some measure of control over $\Hon(K, G)$.

\newcommand{\sG}{{\scriptscriptstyle{(G)}}}

\begin{question} \label{q:main}
Let $G$ be a reductive group defined over a field $K$.
Is there a functorial map
\begin{equation}\label{e:main}
(\xi,n)\mapsto \xi^{\di  n}\hs\colon\,\hs \Hon(K,G)\times \Z\,\to\, \Hon(K,G)
\end{equation}
with the following properties?

\begin{enumerate}
\item[\rm(i)] When $G$ is a torus, the map $\xi\mapsto\xi^{\di n}$
coincided with the natural power map  $\xi\mapsto \xi^n$
induced by the group structure of $\Hon(K, G)$;
\item[\rm(ii)] $(\xi^{\di m})^{\di n} = \xi^{\di mn}$ for every $m, n \in \Z$.
\end{enumerate}

Here ``functorial'' means ``functorial with respect to both homomorphisms $G \to G'$ of reductive $K$-groups
and finite separable extensions $K'/K$\,''.
\end{question}

Note that property (i) together with the functoriality condition
applied to the trivial torus $T=\{1\}\subseteq G$
implies that  $1_\sG^{\di n} = 1_\sG$
for any integer $n$, where $1_\sG$ denotes
the neutral cohomology class in $\Hon(K, G)$.

We also remark that there is a well-understood power map
in the special case where $G$ is the projective linear group $\PGL_n$\hs.
Indeed, the Galois cohomology set $\Hon(K, \PGL_n)$ can be identified with the set of isomorphism
classes of central simple algebras of degree $n$ over $K$, or equivalently, with the subset
of the Brauer group $\operatorname{Br}(K)$ consisting of classes of index dividing $n$.
Recall that the {\em index} of a central simple $K$-algebra $A$ is the  degree of $D$ over $K$,
where $D$ is a division algebra such that $A\simeq M_r(D)$ for some
positive integer $r$.
The set of classes of index dividing $n$ is not a subgroup of $\operatorname{Br}(K)$.
This is because the index of the tensor product $A \otimes_K B$ of central simple $K$-algebras
$A$ and $B$ may be  greater than the indices of $A$ and $B$. For an example of
a field $K$ (not local or global) and division algebras $A$, $B$ of degree (and hence, index) 2 over $K$
whose tensor product is a division algebra of degree (and hence, index) 4 over $K$, see
\cite[Theorem~15.7]{Pierce}.
On the other hand, this set of classes is closed under taking powers
(when the class of $A$ maps to the class of $A^{\otimes d}$).
This is because the index of a power $A^{\otimes d}$
always divides the index of $A$; see~\cite[Proposition 13.4(viii)]{Pierce}.
For background material on central simple algebras and Brauer classes
we refer the readers to \cite[Chapters 12-14]{Pierce}  and~\cite[Chapter 2]{gille-szamuely}.

The answer to Question~\ref{q:main} is trivially ``Yes'' when $K$ is a finite field,
because then by Lang's theorem \cite{lang-56} we have $\Hon(K,G)=1$ for all connected $K$-groups $G$.
Moreover, the answer is ``Yes" when $K$ is a non-archimedean local field,
or a  global  function field, or a totally imaginary number field,
because in these cases the pointed set $\Hon(K,G)$ has a canonical abelian group structure;
see \cite[Theorem 5.5.2(1) and Corollary 5.5.7(3)]{BK}.
Furthermore, this group structure is, in fact, functorial in $G$ and $K$.
See also Labesse \cite[Section 1.6]{Labesse}, where a generalization to {\em quasi-connected} reductive groups is considered,
as well as Kottwitz~\cite[Proposition 6.4]{Kottwitz-Duke} and Gonz\'alez-Avil\'es~\cite[Theorem 5.8(i)]{GA},
where the corresponding abelian group structures are constructed in (different) non-functorial ways.

When $K=\R$ or $K$ is a number field admitting a real embedding, it is impossible to define a group structure
on $\Hon(K, G)$ for all reductive groups $G$ in a functorial way; see  \cite{Borovoi-NG}.
Nevertheless, we prove the following.

\begin{theorem} \label{thm.main1}
For $K=\R$ or any number field $K$,
there exists a unique power operation~\eqref{e:main} over $K$
satisfying the requirements  of Question \ref{q:main}.
\end{theorem}

Thus, it is possible to define a power operation when $K$ is a local or global field.
On the other hand, it is impossible to define a functorial in $G$ power operation
over an arbitrary field:

\begin{theorem}\label{t:non-ex}
There exists a field $K$ and an integer $n$ for which there is no functorial in $G$ power operation
\[\di n\colon \Hon(K,G)\to \Hon(K,G)\]
defined for all reductive $K$-groups $G$ and  having  property (i)
of Question \ref{q:main}.
\end{theorem}

In Section \ref{s:number} we shall deduce this theorem from Proposition~\ref{p:Gille}
in Appendix \ref{app:Gille} written by Philippe Gille.

\begin{remark}
The argument in Appendix B shows that the field $K$ in Theorem \ref{t:non-ex} can be chosen to be of the form
$K=L((z))$, where $L$ is some finite extension of the field
of iterated formal Laurent series $\C((x))((y))$ such that the degree $[L: \C((x))((y))$ is prime to $5$.
\end{remark}

When $K$ is a local or global  field, we use the power operation to
define the {\em period} $\per(\xi)$
of an element of $\xi \in \Hon(K, G)$ to be
the greatest common divisor of the integers $n$ such that $\xi^{\di n} = 1$.
By Corollary \ref{c:xi^per}, the period $\per(\xi)$ is the least integer $n\ge1$ such that $\xi^{\di n}=1$.

Over an arbitrary field $K$, for $\xi\in \Hon(K,G)$
we say that a field extension $L/K$ {\em splits} $\xi$
if the restriction $\Res_{L/K}(\xi)$ equals $1$.
We define the {\em index} $\ind(\xi)$
to be the greatest common divisor of the degrees $[L: K]$
where $L/K$ ranges over all finite field extensions splitting $\xi$.

In the special case where $G = \PGL_n$\hs,
the above notions of the period and index reduce to the period and index of
a central simple algebra; see Example~\ref{ex.PGLn-b}.
(Note that the term ``exponent" is sometimes used synonymously with ``period".)

By Brauer's theorem, see  \cite[Theorem 2.8.7.1]{gille-szamuely}, for any field $K$ and
any $\xi \in \Hon(K,\PGL_n)$,
the period $\per(\xi)$ divides the index $\ind(\xi)$.
We prove a similar assertion for an arbitrary reductive group $G$,
but only over a local or global  field $K$,
because we  define the period  of $\xi \in \Hon(K, G)$
only when $K$ is a local or global field.

\begin{theorem}[Theorem \ref{t:divides}]
\label{thm:divides}
Let $G$ be a reductive group over a local or global field $K$,
and let $\xi \in \Hon(K, G)$. Then $\per(\xi)$ divides $\ind(\xi)$.
\end{theorem}

Let $G$ be a semisimple
group defined over a local or  global   field $K$. We say that
$G$ has {\em the period $2$ property} if $\per(\xi) = 2$ for every $\xi\neq 1$ in  $\Hon(K, G)$.

\begin{theorem}[Theorem \ref{t:2-property}]
\label{t:2-property-intro}
Let $K$ be a  {\emm global}  field, and let $G$ be a semisimple $K$-group.
Then $G$ has the period 2 property if and only if $2\cdot \X^*(\mu)=0$,
where $\X^*$ denotes the character group, $\mu=\ker [G^\ssc\to G]$,
and $G^\ssc$ is the universal cover of $G$.
\end{theorem}

Note that  the ``only if'' assertion of Theorem \ref{t:2-property-intro}, stated for global fields,
fails for non-archimedean local fields; see Remark \ref{r:2-property}.

To motivate the next result of this paper,  we recall the following.

\begin{proposition}
\label{prop.ABHN}
Let $K$ be a local or global field.  Then the period of any Brauer class in $\Br(K)$ equals its index.
\end{proposition}

When $K$ is an archimedean local field, that is,  $\R$ or $\C$, this is obvious. In the case where
$K$ is a non-archimedean local field, Proposition~\ref{prop.ABHN} is proved
in \cite[Section XIII.3,  Corollary 1 of Proposition 7]{Serre-LF}.
In the case where $K$ is a global field, Proposition~\ref{prop.ABHN} is a consequence of the celebrated
theorems of Albert-Brauer-Hasse-Noether and of Grunwald-Wang.
For a proof, see~\cite[Section 18.6]{Pierce} or \cite[Section 5.4.4]{Roquette}
in the case where $K$ is a number field,
and \cite[Theorem 3.6]{eisentrager} when $K$ is an arbitrary global field.

We shall prove a generalization of Proposition~\ref{prop.ABHN} with $\PGL_n$ replaced by
a reductive $K$-group $G$. Let $\pi_1^\alg(G)$ denote the algebraic fundamental group
of $G$ of \cite[Section 1]{Borovoi-Memoir}.
This is a finitely generated abelian group endowed with a natural action of
the absolute Galois group $\Gal(K^s/K)$;
see Section \ref{s:abelian} for the definition of $\pi_1^\alg(G)$.

\begin{theorem}[Theorem~\ref{t:split}]
\label{thm.main2}
Let $G$ be a reductive group over a local or global field $K$ for which
the algebraic fundamental group $M=\pi_1^\alg(G)$ is {\emm split},
that is, the absolute Galois group $\Gal(K^s/K)$ acts on $M$ trivially.
Then $\per(\xi)=\ind(\xi)$ for any $\xi\in \Hon(K,G)$.
\end{theorem}

\smallskip
\begin{example}
The algebraic $K$-group $\PGL_n$ is split,
and therefore its algebraic fundamental group $\pi_1^\alg(\PGL_n)$  is split;
see Remark \ref{r:split}.
Namely, $\pi_1^\alg(\PGL_n)\cong\Z/n\Z$ (with trivial Galois action).
Thus Theorem~\ref{thm.main2} is indeed a generalization of Proposition~\ref{prop.ABHN}.
\end{example}

Without the splitting assumption on the algebraic fundamental group,
Theorem~\ref{thm.main2} may fail.

\begin{theorem}[Theorems \ref{thm.main3-a} and \ref{thm.main3-b}
in Appendix~\ref{sect.example-local}]
\label{thm.main3}
Let $K$ be
\begin{enumerate}
\item[\rm (a)] a non-archimedean local field of residue characteristic not $2$, or
\item[\rm (b)] a global field of characteristic not 2.
\end{enumerate}
Assume that $K$  contains $\sqrt{-1}$.
Then there exist a $6$-dimensional $K$-torus $T$
and a cohomology class $\xi\in \Hon(K,T)$ such that
$ \per(\xi)=2$, but $4\hs\hs|\ind(\xi)$.
\end{theorem}

Now recall that for an arbitrary field $K$, elements $\xi \in \Hon(K, \PGL_n)$
are in functorial bijective correspondence with isomorphism classes
of central simple algebras of degree $n$ over $K$.
As mentioned above, $\ind(\xi)$ and $\per(\xi)$
are the index and the period of a central simple algebra corresponding to $\xi$.
By Brauer's theorem~\cite[Theorem 2.8.7.2]{gille-szamuely},
the integers $\ind(\xi)$ and $\per(\xi)$ have the same prime factors.
Therefore, $\ind(\xi)$ divides $\per(\xi)^d$ for some
positive integer $d$.
We prove a similar assertion for any reductive group $G$,
but once again, only over a local or global field $K$,
because we only {define} the period in this setting.

\begin{theorem}[Theorems \ref{t:n[a]-l} and \ref{t:ind-per-d}]
\label{thm.main4}
Let $G$ be a reductive group over a local or global field $K$.
Then for every $\xi \in \Hon(K, G)$, the index  $\ind(\xi)$
divides $\per(\xi)^d$ for some positive integer $d$.
\end{theorem}

Combining Theorems~\ref{thm:divides} and~\ref{thm.main4}, we obtain:

\begin{corollary}
\label{c:main}
Let $G$ be a reductive group over a local or global field $K$.
Then for every $\xi \in \Hon(K, G)$,
the  integers $\per(\xi)$ and $\ind(\xi)$ have the same prime factors.
\end{corollary}

Theorem~\ref{thm.main4} suggests the following question.

\begin{question} \label{q.period-index}
Let $G$ be a reductive group over a local or global field $K$.
What is the smallest positive integer $d$ such that $\ind(\xi)$ divides $\per(\xi)^d$ for
every $\xi \in \Hon(K, G)$?
\end{question}

Note that Question~\ref{q.period-index}  may be viewed as a variant of the
{\em period-index problem}. In the case of central simple algebras
(that is, in the case where $G = \PGL_n$) over a local or global field $K$, the period-index problem
is settled by Proposition~\ref{prop.ABHN},
which says that $d$ can always be taken to be $1$.
Over other fields, the period-index problem for central simple algebras is open
and remains an active area of research;
for a brief introduction, we refer the reader to~\cite[Section 4]{abgv}.
In the case of a reductive group $G$ over a local or global field $K$, Theorem~\ref{thm.main2} settles
the period-index problem assuming that
the algebraic fundamental group of $G$ is split. For an arbitrary reductive group $G$,
Question~\ref{q.period-index} remains open;
however, we give upper bounds  for $d$. See Corollary~\ref{c:n[a]-l}(b)
in the case of a non-archimedean local field,
and in Remark \ref{rem.upper-bound-on-d} in the case of a global field.

The remainder of this paper is structured as follows.

In Section~\ref{s:abelian} we briefly recall the definition of abelian cohomology
and the abelianization map for a reductive group defined over an arbitrary field $K$.
In Sections~\ref{s:infinite} we discuss the localization map at infinity
in the case where $K$ is a number field.
In Section~\ref{s:Cartesian} we
explain how the abelianization map and the localization map fit together into a Cartesian square.
The constructions of Sections  \ref{s:abelian}--\ref{s:Cartesian} are known; our
presentation highlights their functoriality in both $G$ and $K$.

In Section~\ref{s:real} we define the power operation on $\Hon(K,G)$ in the case
where $K=\R$ is the field of real numbers; see Definition \ref{d:real}.
In Section \ref{s:number} we define the power operation on $\Hon(K,G)$
where $K$ is a non-archimedean local field or a global function field (see Definition \ref{d:nal-gf})
and in the main case where $K$ is a number field (see Construction~\ref{con:power}).
The remainder of Section~\ref{s:number} is devoted to
proving Theorem~\ref{thm.main1} and investigating further uniqueness and functorial properties
of our power operation. We also deduce Theorem \ref{t:non-ex}
from Proposition \ref{p:Gille}.

In Section \ref{s:explicit} we express the power operation over a global field $K$
in terms of the explicit description of $\Hon(K,G)$ given in \cite{BK}.
In Section~\ref{s:period} we use the power operation
to define  the period and the index of an element of $\Hon(K, G)$
in the case where $G$ is a reductive group over a local or global field
$K$, and we prove Theorem~\ref{thm:divides}.
In Section \ref{sect.examples} we discuss the period 2 property
and prove Theorem \ref{t:2-property-intro}.

After some preparatory work in Section~\ref{s:transfer},
in Section~\ref{sect.local} we prove Theorem~\ref{thm.main4} in the local case.
Theorem~\ref{thm.main2} is proved in Section~\ref{sect.main2}.
Finally, in Section ~\ref{sect.main4-0} we prove Theorem ~\ref{thm.main4} in the global case.

In Appendix ~\ref{sect.example-local},
the first-named author proves Theorem~\ref{thm.main3}.
In Appendix \ref{app:Gille}, Philippe Gille proves Proposition~\ref{p:Gille},
from which Theorem \ref{t:non-ex} follows.

\subsection*{Acknowledgments}
We are grateful to the referee for quick and thorough reading,
and for constructive comments which helped us to improve the exposition.

We thank Benjamin Antieau and Benedict Williams
for stimulating conversations about Adams operations
which led to Question~\ref{q:main}.
We are grateful to Kasper Andersen, Tyler Lawson, and Loren Spice
for answering Borovoi's questions in MathOverflow.
We thank Victor Abrashkin, Ofer Gabber, David Harari,
Boris Kunyavski\u{\i}, Andrei S. Rapinchuk, Zev Rosengarten,  Loren Spice,
Nguy\^{e}\~{n} Qu\^{o}\'{c} Th\v{a}\'{n}g, and Angelo Vistoli
for stimulating discussions and email correspondence.
Special thanks to Philippe Gille for writing Appendix \ref{app:Gille}.

This paper was conceived during the first-named author’s visit
to the University of British Columbia, and
written, in part, during his visits to the Max-Planck-Institut f\"ur Mathematik, Bonn,
and Institut des Hautes \'Etudes Scientifiques (IHES).
He thanks these institutions  for their hospitality, support, and
excellent working conditions.

\subsection*{Notation}
For a field $K$, we fix  an algebraic closure $\ov K$ of $K$
and let $K^s$ denote the separable closure of $K$ in $\ov K$.
In this paper, all finite extensions of $K$ that we consider, are contained in $\ov K$.
Similarly, all finite {\em separable} extensions of $K$ that we consider, are contained in $K^s$.
Therefore, for two finite separable extensions $F/K$ and $L/K$,
we may consider the intersection $F\cap L$ and the composite $FL$.

\setcounter{section}{1}

\section{The algebraic fundamental group and abelian Galois cohomology}
\label{s:abelian}

In this section,  $G$ is a reductive group over an arbitrary field $K$.
We recall the definitions of the algebraic fundamental group $\pi_1^\alg(G)$ and the
first abelian Galois cohomology group $\Ho^1_\ab(K,G)$.
We  show that $\Ho^1_\ab(K,G)$ depends functorially  on $G$ and  $K$.
We refer to \cite[Section 1]{Borovoi-Memoir} and \cite[Chapter 5]{BK} for details;
see also \cite[Section 5]{GA}.

Let $G$ be a reductive group over a field $K$.
We recall the definition of $\pi_1^\alg(G)$ from \cite{Borovoi-Memoir}.
Let $G^\ssc$ denote the universal cover of the derived subgroup $[G,G]$ of $G$;
see Borel and Tits~\cite[Proposition (2.24)(ii)]{Borel-Tits72}
or Conrad, Gabber and Prasad~\cite[Corollary A.4.11(1)]{CGP}.
The group $G^\ssc$ is simply connected, which explains the superscript $^\ssc$.
We consider the composite homomorphism
\[\rho\colon G^\ssc\onto [G,G]\into G.\]
For a maximal torus  $T\subseteq G$, we denote
$T^\ssc=\rho^{-1}(T)\subseteq G^\ssc$\hs,
which is a maximal torus of $G^\ssc$.
Let $\X_*(T)$ denote the group of cocharacters of $T$ defined over $K^s$.
Since the kernel $\ker[\rho\colon T^\ssc\to T]$ is finite,
the induced homomorphism on cocharacters $\rho_*\colon \X_*(T^\ssc)\to \X_*(T)$ is injective.
The {\em algebraic fundamental group} of $G$ is defined by
\[\pi_1^\alg(G)=\X_*(T)/\rho_*\X_*(T^\ssc).\]
In particular, if $G=T$ is a torus, then $\pi_1^\alg(T)=\X_*(T)$.

The Galois group $\G=\Gal(K^s/K)$ naturally acts on $\pi_1(G)$,
and the  $\Gal(K^s/K)$-module $\pi_1^\alg(G)$
is well defined (does not depend on the choice of $T$
up to a transitive system of isomorphisms);
see \cite[Lemma 1.2]{Borovoi-Memoir} and \cite[Proposition 2.2.1]{Fu}.

\begin{proposition} \label{prop.ss}
Let $G$ be a {\emm semisimple} group defined over a field $K$.
Let $M=\pi_1^\alg(G)$ be the algebraic fundamental group of $G$,
and let  $\mu=\ker[G^\ssc\to G]$
(which is a finite group of multiplicative type, not necessarily smooth).
Let $\X^*(\mu)$ denote the character group of $\mu$.
Then there is a canonical isomorphism of Galois modules
\[M\cong\Hom\big(\X^*(\mu),\Q/\Z\big).\]
\end{proposition}

Our proof of Proposition~\ref{prop.ss} will rely on the following elementary lemma.

\begin{lemma}
\label{l:duality}
Let $B$ be a lattice (that is, a finitely generated free abelian group),
and let $\iota\colon A\hookrightarrow B$ be a sublattice of finite index.
Consider the dual lattices
$$ A^\vee=\Hom(A,\Z),\quad\ B^\vee=\Hom(B,\Z),$$
and the canonical injective homomorphism
$\iota^\vee\colon B^\vee\hookrightarrow A^\vee$.
Then there is a canonical perfect pairing of finite abelian groups
$$A^\vee/B^\vee\times B/A\to \Q/\Z$$
where we identify $B^\vee$ with $\iota^\vee(B^\vee)\subseteq A^\vee$
and  $A$ with $\iota(A)\subseteq B$.
\end{lemma}

\begin{proof} See ~\cite{Borovoi23-MO}. \end{proof}

\begin{proof}[Proof of Proposition \ref{prop.ss}]
Let $T\subset G$ be a maximal torus, and write $T^\ssc=\rho^{-1}(T)\subset G^\ssc$.
Then by the definition of the algebraic fundamental group $M=\pi_1^\alg(G)$ we have
\[M=\X_*(T)/\X_*(T^\ssc).\]
On the other hand, from the short exact sequence
\[1\to\mu\to T^\ssc\to T\to 1\]
we obtain a short exact sequence
\[0\to\X^*(T)\to\X^*(T^\ssc)\to \X^*(\mu)\to 0;\]
see Milne \cite[Theorem 12.9(b)]{Milne}.
It follows that there is a canonical isomorphism
\[ \X^*(\mu)=\X^*(T^\ssc)/\X^*(T)=\X_*(T^\ssc)^\vee/X_*(T)^\vee.\]
Now by Lemma \ref{l:duality} we have a Galois-equivariant perfect pairing
\[\X^*(\mu)\times M\to\Q/\Z,\]
and the proposition follows.
\end{proof}

\begin{example}
Let $G=\PGL_n$\hs.   Then $\mu\cong\mu_n$, \,$\X^*(\mu)\cong\Z/n\Z$,  \,and
\[\pi_1^\alg(G)\cong \Hom(\X^*(\mu), \Q/\Z)\cong \Hom(\Z/n\Z, \Q/\Z)\cong \tfrac1n \Z/\Z\cong\Z/n\Z.\]
\end{example}

We recall the definition of  the {\em abelian Galois cohomology}
$\Ho^1_\ab(K,G)$; see \cite[Chapter 5]{BK} for details.
By definition,
\[\Ho^1_\ab(K,G)=\HH^1(K,G^\ssc\to G)=\Zl^1(K,G^\ssc\to G)/\sim\]
where $\Zl^1(K,G^\ssc\to G)$ is the set of 1-hypercocycles
and $\sim$ is a certain equivalence relation.
A 1-hypercocycle is a pair $(c^\ssc,c)$
of locally constant maps
\begin{align*}
c^\ssc\colon \G_K\times \G_K\to G^\ssc(K^s),
\quad\ c\colon \G_K\to G(K^s)
\end{align*}
where $\G_K=\Gal(K^s/K)$ and where the pair $(c^\ssc,c)$
satisfies the 1-hypercocycle conditions of \cite[Section 3.3.2]{Borovoi-Memoir}.

The inclusion map $i\colon T\into G$ induces a morphism of complexes
\[ (T^\ssc\to T)\hs\to\hs (G^\ssc\to G),\]
which is a {\em quasi-isomorphism}: the induced homomorphisms
of the kernels and cokernels are isomorphisms.
It follows that the map on hypercohomology
\[\HH^1(K,T^\ssc\to T)\to \HH^1(K,G^\ssc\to G)\]
is bijective; see \cite[Proposition 5.6]{Noohi}.
The set $\HH^1(K,T^\ssc\to T)$ has a natural structure of an abelian group,
and this abelian group depends only on $\pi_1^\alg(G)$.
Thus we obtain a structure of an abelian group on $\Ho^1_\ab(K,G)$.
This group structure can be described in term of $G$ only
(not mentioning  $T$); see \cite[Section 5.4]{BK}.
We obtain a well-defined abelian group $\Ho^1_\ab(K,G)$
(depending only on $\pi_1^\alg(G)$).

\begin{remark}
If $G$ is a torus, then $T=G$ and $\Ho^1_\ab(K,G)=\Hon(K,G)$.
If $G$ is semisimple and $\mu=\ker\big[G^\ssc\to G\big]$,
then we have a canonical isomorphism
\[ \Ho^1_\ab(K,G)\isoto\Ho^2_{\rm fppf}(K,\mu),\]
where $\Ho^2_{\rm fppf}(K,\mu)$ denotes the second flat cohomology
of the finite group scheme $\mu$ (which may not be smooth).
In particular, if $G$ is semisimple and simply connected, then $\mu=1$ and  we have $\Ho^1_\ab(K,G)=1$.
\end{remark}

Let $\varphi\colon G\to H$ be a homomorphism of reductive $K$-groups.
It induces a homomorphism $\varphi^\ssc\colon G^\ssc\to H^\ssc$
and thus a morphism of pairs
\[(\varphi^\ssc,\varphi)\colon\,(G^\ssc\to G)\to(H^\ssc\to H).\]
By definition, the induced map
\[\varphi_\ab\colon \Ho^1_\ab(K,G)\to\Ho^1_\ab(K,H)\]
sends class of a hypercocycle $(c^\ssc\to c)$ to the class of
\[(\varphi^\ssc\circ c^\ssc, \varphi\circ c)\hs\in\Zl^1(K,H^\ssc\to H).\]
It is easy to see that $\varphi_\ab$ is a homomorphism of abelian groups.

Let $\iota\colon K\into  L$ be a field extension.
Choose a separable closure $L^s$ of $L$,
and let $K^s\subseteq L^s$ denote the separable closure of $K$ in $L^s$.
We obtain an induced homomorphism of absolute Galois groups
\[\iota^*\colon\G_L\coloneqq \Gal(L^s/L)\to \Gal(K^s/K)\eqqcolon\G_K\hs.\]
We define a homomorphism
\[\iota_\ab\colon \Ho^1_\ab(K,G)\to \Ho^1_\ab( L,G)\]
by sending the class of a hypercocycle $(c^\ssc,c)\in \Zl^1(K, G^\ssc\to G)$
to the class of
\[ \big(c^\ssc\circ(\iota^*\times\iota^*),\hs c\circ\iota^*\big)\in \Zl^1( L, G^\ssc\to G).\]
Here $c^\ssc\circ(\iota^*\times\iota^*)$ is the map
that sends $(\gamma_1,\gamma_2)\in\G_ L\times\G_L$ to
\[ c^\ssc\big(\iota^*(\gamma_1),\iota^*(\gamma_2)\big)\in G^\ssc(K^s)\subseteq G^\ssc(L^s).\]
This map $\iota_\ab$ is well defined.
It is easy to see that  $\iota_\ab$ is a homomorphism of abelian groups.

With the above definitions of $\varphi_\ab$ and $\iota_\ab$\hs,
the correspondence  $(K,G)\ \functor\, \Ho^1_\ab(K,G)$ is a functor in $G$ and $K$.

Following \cite[Section 3]{Borovoi-Memoir} and \cite[Chapter 5]{BK}, we define the abelianization map
\[\ab\colon \Hon(K,G)\to \Ho^1_\ab(K,G).\]
This map sends the class of a cocycle $c\in\Zl^1(K,G)$
to the class of $(1,c)\in\Zl^1(K,G^\ssc\to G)$.

\begin{proposition}\label{p:ab}
The map
\[ \ab\colon \Hon(K,G)\to \Ho^1_\ab(K,G)\]
where $G$ is a reductive $K$-group,
is functorial in both $G$ and in $K$. That is,

\begin{enumerate}

\item[\rm (1)]
The following diagram, induced by a homomorphism $\varphi\colon G\to H$
of reductive $K$-groups, commutes:

\[
\xymatrix@C=12mm{
\Hon(K,G)\ar[r]^-\ab\ar[d]_-{\varphi_*}   &\Ho^1_\ab(K,G)\ar[d]^-{\varphi_*}\\
\Hon(K,H)\ar[r]^-\ab                      &\Ho^1_\ab(K,H)
}
\]

\smallskip
\item[\rm (2)]
The following diagram, induced by a field extension $\iota\colon K\into L$, commutes:
\[
\xymatrix@C=12mm{
\Hon(K,G)\ar[r]^-\ab\ar[d]_-{\iota_*}   &\Ho^1_\ab(K,G)\ar[d]^-{\iota_*}\\
\Hon(L,G)\ar[r]^-\ab                    &\Ho^1_\ab(L,G)
}
\]
\end{enumerate}

\end{proposition}

The proof of Proposition~\ref{p:ab} is straightforward; we leave it as an exercise for the reader.

\section{The localization map at infinity}
\label{s:infinite}

Let $K$ be a number field and $G$ be a linear algebraic group over $K$.
Let $K_v$ denote the completion of $K$ at a place $v$.
Then we have a field extension $\iota_v \colon K \into K_v$
and the induced {\em localization map}
\[\loc_v=\iota_{v,*}\colon\, \Hon(K,G)\to\Hon(K_v,G).\]
Now let $\V_\infty(K)$ denote the set of infinite (that is, archimedean) places of $K$.
We set
\[ \Hon(K_\infty, G)=\!\!\prod_{v\in\V_\infty(K)}\!\! \Hon(K_v,G) \]
and combine the localization maps $\loc_v$ for $v \in \V_\infty(K)$ into a single map,
the {\em localization map at infinity}
\begin{equation}\label{e:loc-inf}
\loc_\infty\colon\  \Hon(K,G)\to \Hon(K_\infty,G),\qquad \xi\,\mapsto
   \Big( \loc_v(\xi)\Big)_{v\in\V_\infty(K)}\in \!\! \prod_{v\in \V_\infty(K)}\!\! \Hon(K_v,G)\,=\,\Hon(K_\infty,G).
\end{equation}

A homomorphism of linear algebraic $K$-groups $\varphi\colon G\to H$
induces a map (a morphism of pointed sets)
\[\varphi_{\infty}\colon  \Hon(K_\infty, G)\to  \Hon(K_\infty, H),\]
which is the product over $v\in \V_\infty(K)$  of the maps $\varphi_v\colon \Hon(K_v,G)\to\Hon(K_v,H)$
where $\varphi_v$ sends the class of $c_v\in\Zl^1(K_v,G)$ to the class of $\varphi\circ c_v\in \Zl^1(K_v,H)$.
Note that for any $v\in\V(K)$, the following diagram commutes:
\begin{equation}\label{e:varphi-loc_v}
\begin{aligned}
\xymatrix@C=12mm{
\Hon(K,G)\ar[r]^-{\varphi_*}\ar[d]_-{\loc_v}  &\Hon(K,H)\ar[d]^-{\loc_v} \\
\Hon(K_v,G)\ar[r]^-{\varphi_v}  &\Hon(K_v,H) .
}
\end{aligned}
\end{equation}

Let $\iota\colon K\to L$ be a finite extension of $K$.
We define an induced map (a morphism of pointed sets)
\[\iota_\infty\colon \Hon(K_\infty, G)\to \Hon(L_\infty,G)\]
as follows.

For a place $v\in \V_\infty(K)$ and a place $w\in\V_\infty(L)$ over $v$,
consider the commutative diagram of fields
\[
\xymatrix@C=12mm{
K\ar[r]^-\iota\ar[d] &L\ar[d]\\
K_v\ar[r]         &L_w
}
\]
and the induced commutative diagram in cohomology
\begin{equation}\label{e:loc-v-w}
\begin{aligned}
\xymatrix@C=12mm{
\Hon(K,G)\ar[r]^-{\iota_*}\ar[d]_{\loc_v} &\Hon(L,G)\ar[d]^{\loc_w}\\
\Hon(K_v,G)\ar[r]^-{\iota_{v,w}} &\Hon(L_w,G)
}
\end{aligned}
\end{equation}
Note that, if $L_w\simeq \C$, then $\Hon(L_w,G)=1$ and the map $\iota_{v,w}$ sends
all elements of $\Hon(K_v,G)$ to $1$.
If $L_w\simeq \R$, then also $K_v\simeq\R$
and the map $\iota_{v,w}$ is bijective.

We combine the maps $\iota_{v, w}$ into a single map
\begin{align*}
\iota_{\infty} \colon \ \Hon(K_\infty,G) =\prod_{v\in \V_\infty(K)}\Hon(K_v,G)
 \ &\lra    \prod_{w\in \V_\infty(L)}\Hon(L_w,G)=\Hon(L_\infty,G),\\
   \quad\big(\xi_v\big)_{v\in\V_\infty(K)}\ &\longmapsto\  \big(\iota_{v(w),w}\hs(\xi_v)\big)_{w\in \V_\infty(L)}
\end{align*}
where $v(w)$ denotes the restriction of $w$ to $K$.

\begin{proposition}\label{p:loc}
The map $\loc_\infty$ of \eqref{e:loc-inf}
is functorial in both $G$ and $K$. That is,
\begin{enumerate}
\item[\rm (1)]
For any number field $K$, the diagram
\[
\xymatrix@C=12mm{
\Hon(K,G)\ar[r]^-{\varphi_*}\ar[d]_-{\loc_\infty}  &\Hon(K,H)\ar[d]^-{\loc_\infty} \\
\Hon(K_\infty, G)\ar[r]^{\varphi_\infty}              &\Hon(K_\infty,H);
}
\]
induced by a homomorphism $\varphi\colon G\to H$ of reductive $K$-groups,
commutes.

\smallskip
\item[\rm (2)] For any embedding of number fields $\iota\colon K\into L$,
and any reductive $K$-group $G$, the diagram
\[
\xymatrix@C=12mm{
\Hon(K,G)\ar[r]^{\iota_*}\ar[d]_-{\loc_\infty}   &\Hon(L,G)\ar[d]^-{\loc_\infty}\\
\Hon(K_\infty,G)\ar[r]^{\iota_{\infty}}                 &\Hon(L_\infty,G)
}
\]
induced by $\iota$, commutes.
\end{enumerate}
\end{proposition}

\begin{proof}
Assertion (1) follows from the commutativity of  diagram \eqref{e:varphi-loc_v},
and assertion (2) follows from the commutativity of  diagram \eqref{e:loc-v-w}.
\end{proof}

\section{The abelianization map and a Cartesian square}
\label{s:Cartesian}

\begin{lemma}
Let $G$ be a reductive group over a number field $K$.
Then the following diagram commutes:
\begin{equation}\label{e:Cartesian}
\begin{aligned}
\xymatrix@C=8mm@R=10mm{
\Ho^1(K,G) \ar[r]^-{\ab} \ar[d]_-{\loc_\infty}    &\Ho^1_\ab(K,G)\ar[d]^-{\loc_\infty}\\
\Ho^1(K_\infty,G) \ar[r]^-{\ab}                   &\Ho^1_\ab(K_\infty,G) ,
}
\end{aligned}
\end{equation}
where
\[ \Ho^1(K_\infty,G)=\hskip-3mm\prod_{v\in \V_\infty(K)}\hskip-3mm\Ho^1(K_v,G)
     \qquad \text{and} \qquad
    \Ho^1_\ab(K_\infty,G)=\hskip-3mm\prod_{v\in \V_\infty(K)}\hskip-3mm
              \Ho^1_\ab(K_v,G)\]
\end{lemma}

\begin{proof}
It suffices to show that for any $v\in \V_\infty(K)$, the following diagram commutes:
\[
\xymatrix@C=8mm@R=8mm{
\Ho^1(K,G) \ar[r]^-{\ab} \ar[d]_-{\loc_v}    &\Ho^1_\ab(K,G)\ar[d]^-{\loc_v}\\
\Ho^1(K_v,G) \ar[r]^-{\ab}                   &\Ho^1_\ab(K_v,G) .
}
\]
This follows from Proposition \ref{p:ab}(2) applied to the field extension $\iota_v\colon K\into K_v$\hs.
\end{proof}

\begin{theorem}[\hs{\cite[Theorem 5.11]{Borovoi-Memoir}}\hs]
\label{t:cd}
Let $G$ be a reductive group over a number field $K$. Then
the commutative  diagram \eqref{e:Cartesian} is a Cartesian square,
that is, it identifies the pointed set $\Ho^1(K,G)$ with the fibered product
of  $\Ho^1_\ab(K,G)$
and   $\Hon(K_\infty,G)$
over  $\Ho^1_\ab(K_\infty,G)$.
\end{theorem}

Theorem \ref{t:cd} says that one can identify $\Hon(K,G)$ with the set of pairs
\[\Big\{(\xi_\ab,\hs\xi_\infty)\in \Ho^1_\ab(K,G)\times \Hon(K_\infty,G)\ \Big|\
\loc_\infty(\xi_\ab)=\ab(\xi_\infty)\Big\}.\]

\begin{remark}
All arrows in  diagram \eqref{e:Cartesian} are surjective.
Both horizontal arrows are surjective
by \cite[Theorems 5.7 and 5.4]{Borovoi-Memoir}.
The left-hand vertical arrow  is surjective by \cite[Proposition 6.17 on p.~337]{PR}.
This tells us that the right-hand vertical arrow is also surjective because
diagram~\eqref{e:Cartesian} is commutative.
\end{remark}

\begin{remark}
One readily checks that the commutative diagram \eqref{e:Cartesian}
is functorial in both $G$ and $K$. Since we shall not need this fact, we
leave the details of the proof to an interested reader.
\end{remark}

\section{Real Galois cohomology}
\label{s:real}

In this section we define the power (diamond) operation
for reductive $\R$-groups, thus proving Theorem \ref{thm.main1} for $K=\R$.

Let $(X,1)$ be a pointed set where $1=1_X\in X$
is the distinguished point.
For $x\in X$, consider the subset $\langle \, x \, \rangle \coloneqq \{ 1, \, x  \}$ of $X$.
This subset has a unique group structure, with $1$ being the identity element.
This group has order $1$ when $x = 1$, and order $2$ if $x \neq 1$.
Let $x^{\dit n}$ denote the $n$-th power of $x$ in the group $\langle \, x \, \rangle$.
That is,
 \[
 x^{\dit   n}=
 \begin{cases}
 x &\text{when $n$ is odd,}\\
 1 &\text{when $n$ is even.}
 \end{cases}
 \]
Clearly, we have
\begin{equation*}
x^{\dit m} \ast x^{\dit n} := x^{\dit (m + n)} \quad {\rm and} \quad
x^{\dit  nm}=(x^{\dit  n})^{\dit  m}
\end{equation*}
for any $m, n \in \Z$, where $*$ denotes multiplication in $\langle x\rangle$.

\begin{lemma}\label{l:morphism}
The operation $\nabla$ is functorial:
if $f\colon (X,1_X)\to (Y,1_Y)$ is a morphism of pointed sets, then we have
$f(x^{\dit  n})=f(x)^{\dit  n}$ for all $x\in X,\ n\in \Z$.
\end{lemma}

\begin{proof}
If $n$ is even, then
\[f(x^{\dit n})=f(1_X)=1_Y=f(x)^{\dit n}.\]
If $n$ is odd, then
\[f(x^{\dit n})=f(x)=f(x)^{\dit n}.\qedhere\]
\end{proof}

Let $A$ be a group {\em of exponent dividing $2$}, that is, a group such that $a^2=1_A$ for all $a\in A$.
It is well known that then $A$ is abelian.
 Indeed, for $a,b\in A$, we have \[ab =(ab)^{-1}=b^{-1}a^{-1}=ba.\]

\begin{lemma}\label{l:exp2}
Let $A$ be a group of exponent dividing $2$.
If we regard $A$ as a pointed set with distinguished point $1_A$, then
\[ a^{\dit  n}=a^n\quad\ \text{for all}\  \,a\in A,\ n\in\Z.\]
\end{lemma}

\begin{proof}
We  compute both sides directly. If $n$ is even, then $a^{\dit n}=1_A=a^n$. If $n$ is odd,
then $a^{\dit n}=a=a^n$.
\end{proof}

\begin{definition}\label{d:real}
Let $G$ be a reductive $\R$-group.
We define the power (diamond) operation  the pointed set $\Hon(\R,G)$
\[
(\xi,n)\mapsto \xi^{\di n}\coloneqq \xi^{\dit n}\quad\ \text{for}\ \,\xi\in \Hon(\R,G),\ n\in\Z
\]
to be the operation $\dit$ defined above.

It is easy to see that this operation satisfies all of the requirements of Question \ref{q:main}.
In particular, functoriality in $G$ follows from Lemma~\ref{l:morphism}. If $G=T$ is an $\R$-torus,
then $\Hon(\R,T)$ is a group of exponent dividing $2$,
and  Lemma \ref{l:exp2} tells us that
\[\xi^{\di n}\coloneqq\xi^{\dit n}=\xi^n \quad\ \text{for all}\ \ \xi\in \Hon(\R,T),\ n\in\Z.\qedhere \]
\end{definition}

\begin{lemma}\label{l:mz}
Let $G$ be a  reductive $\R$-group. Then
\[\ab(\xi^{\di n})=(\ab\,\xi)^n \quad\ \text{for all}\ \,\xi\in \Hon(\R,G),\ n\in \Z.\]
\end{lemma}

\begin{proof}
Let $T\subseteq G$ be a fundamental torus,
that is, a maximal torus containing a maximal compact torus.
Then by Kottwitz~\cite[Lemma 10.2]{Kottwitz-86},
the natural map $\Hon(\R,T)\to\Hon(\R,G)$ is surjective;
see also \cite[Theorem 3.1]{Borovoi-CiM}.
This permits one to reduce the lemma to the case of a torus,
where it becomes obvious.
\end{proof}

\section{The power map: construction, functoriality and uniqueness}
\label{s:number}

In this section we define the power (diamond) operation
for reductive $K$-groups where $K$ is a non-archimedean local field or a global field,
thus proving Theorem \ref{thm.main1} for these fields.
We also prove Theorem \ref{t:non-ex} that shows
that it is impossible to define a diamond operation
over an arbitrary field.

Let $G$ be a reductive group over a local or global field  $K$.
Let $\xi\in \Hon(K,G)$.
When $K=\R$, we defined the power operation $\di$ in Sections \ref{s:real}.

When $K$ is a non-archimedean local field or a  global  function field,
the abelianization map
\[\ab\colon \Hon(K,G)\to\Ho^1_\ab(K,G) \]
is bijective by~\cite[Corollary 5.4.1]{Borovoi-Memoir}
and~\cite[Theorem 5.8(i)]{GA}; see also~\cite[Theorem 5.5.2(1) and Corollary 5.5.7(3)]{BK}.

\begin{definition}\label{d:nal-gf}
When $K$ is a non-archimedean local field or a  global  function field,
we define the power map
\[(\xi,n)\mapsto \xi^{\di  n}\colon \Hon(K,G)\times\Z\to\Hon(K,G)\]
by the formula
\[ \ab(\xi^{\di n})=\ab(\xi)^n\quad\ \text{for}\ \, \xi\in\Hon(K,G),\  n\in\Z.\]
\end{definition}

\begin{lemma}\label{l:eq}
Let $G$ be a reductive group over a number field $K$.
Let $\xi\in \Hon(K,G)$.
Write
\[\xi_{\infty}=\locinf(\xi)\in\Hon(K_\infty,G)\quad\ \text{and} \quad\ \xi_{\abb}=\abb(\xi)\in \Ho^1_\ab(K,G).\]
Then
\begin{equation}\label{e:loc-ab}
\abb \big( (\xi_{\infty})^{\di n} \big) =\locinf\! \big( (\xi_{\abb})^n \big).
\end{equation}
for any $n \in \Z$.
\end{lemma}

Here  $(\xi_{\infty})^{\di n}$ denotes $\xi_{\infty}$
raised to the power $n$ in the pointed set $\prod_{v|\infty} \Hon(K_v,G)$
as defined in Section~\ref{s:real}. By
$(\xi_{\abb})^n$ we mean the $n$-th power of $\xi_{\abb}$ in the abelian group $\Ho^1_\ab(K,G)$.

\begin{proof}
By Lemma~\ref{l:mz} we have
\begin{equation}\label{e:ab-inf}
\abb\big((\xi_{\infty})^{\di n}\big)=\big(\abb(\xi_{\infty})\big)^n.
\end{equation}
Since diagram~\eqref{e:Cartesian} in Section \ref{s:Cartesian} commutes, we have
\begin{equation} \label{e.inf-inf}
\abb(\xi_{\infty}) =
\locinf(\xi_{\abb}). \end{equation}
Moreover, since $\loc_{\infty} \colon \Ho^1_\ab(K,G)
\longrightarrow \prod_{v|\infty}\Ho^1_\ab(K_v,G)$
is a homomorphism of abelian groups, we have
\begin{equation}\label{e:inf-ab}
\big(\locinf(\xi_{\abb})\big)^n =
\loc_\infty\big((\xi_{\abb})^n\big).
\end{equation}
Now~\eqref{e:loc-ab} follows from \eqref{e:ab-inf},~\eqref{e.inf-inf} and \eqref{e:inf-ab}.
\end{proof}

\begin{construction}\label{con:power}
\label{constr:mz}
Let $G$ be a reductive group over a number field $K$.
We define the {\em power map}
\[(\xi,n)\mapsto \xi^{\di  n}\,\colon\, \Hon(K,G)\times \Z\to \Hon(K,G)\]
as follows.
For $\xi\in \Hon(K,G)$ and $n\in\Z$,
let $(\locinf\xi)^{\di n}$ and $(\abb\xi)^n$ be as in Lemma \ref{l:eq}.
By this lemma, equality~\eqref{e:loc-ab} holds.
Since by Theorem \ref{t:cd},  diagram
\eqref{e:Cartesian} is Cartesian, we see that
there exists a unique element
$\eta\in\Hon(K,G)$ such that $\locinf\eta=(\locinf\xi)^{\di n}$
and $\abb\eta=(\abb\xi)^n$. We set $\xi^{\di n}=\eta$.

This construction also works when $K$ is a  global  function field.
Then we have $\V_\infty(K)=\varnothing$,
and similarly to the case of a totally imaginary number field,
we have  $\loc_\infty(\xi)=1$. For a  global  function field, the
power map obtained through this construction coincides with
the power map given by Definition \ref{d:nal-gf}.
\end{construction}

\begin{example}\label{ex:torus}
Assume that $G$ is abelian, that is, $G=T$ is a torus.
Then the abelianization map $\abb\colon \Hon(K,T)\to\Ho^1_\ab(K,T)=\Hon(K,T)$ is the identity map,
whence $\xi^{\di n}=\xi^n$ for all $\xi\in \Hon(K,T)$ and $n\in\Z$.
\end{example}

In the remainder of this section we will explore the functorial and uniqueness properties of the power operation defined in Construction~\ref{con:power}. In particular, we will prove functoriality in $G$ in Proposition~\ref{d6}, functoriality in $K$ in Proposition~\ref{p:ab-n}(i) and uniqueness in
Proposition~\ref{p:uniqueness}.
This, in combination with Example~\ref{ex:torus} and Proposition~\ref{p:ab-n}(v), will complete the proof of Theorem~\ref{thm.main1}. Note that we will also establish some stronger functorial properties of our power operation, beyond those required by Question~\ref{q:main}.

\begin{proposition}
\label{d6}
For reductive groups $G$ over a  global  field $K$,
the power operation $\di$ is functorial in $G$.
In other words, for any homomorphism of reductive $K$-groups $\varphi\colon G\to H$,
any $\xi\in \Hon(K,G)$, and any $n\in \Z$, we have
\[\varphi_*(\xi^{\di n})=\varphi_*(\xi)^{\di n}.\]
\end{proposition}

\begin{proof}
We need to establish the following equalities:
\begin{align}
\ab\big(\varphi_*(\xi^{\di n})\big)&=\ab\big(\varphi_*(\xi)^{\di n}\big), \quad\ \text{and} \label{f6-ab}\\
\loc_\infty\big(\varphi_*(\xi^{\di n})\big)&=\loc_\infty\big(\varphi_*(\xi)^{\di n}\big).\label{f6-loc}
\end{align}

To prove \eqref{f6-ab},
write $\varphi_\ab$ for the induced map $\Ho^1_\ab(K,G)\to\Ho^1_\ab(K,H)$.
We have
\begin{align*}
\ab\big(\varphi_*(\xi^{\di n})\big)&=\varphi_\ab\big(\ab(\xi^{\di n})\big)\quad\ \text{by Proposition \ref{p:ab}(1)}\\
   &=\varphi_\ab(\ab(\xi)^n)   \quad\ \ \,\text{by the definition of the power map  $\di n$}\\
   &=\varphi_\ab(\ab(\xi))^n     \quad \text{because $\varphi_\ab$ is a group homomorphism}\\
   &=\ab(\varphi_*(\xi))^n   \quad\ \text{by Proposition \ref{p:ab}(1)}\\
   &=\ab\big(\varphi_*(\xi)^{\di n}\big)   \quad\ \, \text{by the definition of the power map  $\di n$}.
\end{align*}

To prove \eqref{f6-loc},
write $\varphi_\infty$ for the induced map $\Hon(K_\infty,G)\to\Hon(K_\infty, H)$.
We have
\begin{align*}
\loc_\infty\big(\varphi_*(\xi^{\di n})\big)
   &= \varphi_\infty\big(\loc_\infty(\xi^{\di n})\big)\quad\ \text{by Proposition \ref{p:loc}(1)}\\
   &=\varphi_\infty\big(\loc_\infty(\xi)^{\di n}\big)
      \quad\ \text{by the definition of the power map  $\di n$}\\
   &=\varphi_\infty(\loc_\infty\,\xi)^{\di n}\quad\, \text{because $\varphi_\infty$ is a morphism of pointed sets} \\
   &=\loc_\infty(\varphi_*(\xi))^{\di n}\quad\hm\text{by Proposition \ref{p:loc}(1)}\\
   &=\loc_\infty(\varphi_*(\xi)^{\di n}\big)\quad\ \ \text{by the definition of the power map  $\di n$},
\end{align*}
as desired.
\end{proof}

\begin{proposition} \label{p:ab-n}
Let $G$ be a reductive group over a  global field $K$,
$\xi \in \Hon(K, G)$ and $n \in \Z$.
\begin{enumerate}

\item[\rm (i)] The operation $\di$ is functorial in $K$
with respect to finite separable extensions of  global  fields.
In other words, for any  finite separable
extension of  global  fields $\iota\colon K\into L$, we have
\begin{equation*}
\iota_*(\xi^{\di n})=\iota_*(\xi)^{\di n}.
\end{equation*}

\item [\rm (ii)]
For every place $v$ of $K$, we have
\[\loc_v(\xi^{\di n})=(\loc_v\hs\xi)^{\di n}\hs;\]
see Definition~\ref{d:nal-gf} for the meaning of\/ $(\loc_v\hs\xi)^{\di n}$ when $v$ is non-archimedean.

\item[\rm (iii)]
For any short exact sequence of $K$-groups
$1\to A \to G'\to G\to 1$
where $G'$ and $G$ are reductive and $A$ is central in $G'$ and smooth, we have
\[\delta(\xi^{\di n})=\delta(\xi)^n. \]
Here $\delta\colon \Hon(K,G)\to \Ho^2(K,A)$ is the connecting map.\smallskip

\item[\rm(iv)]
$\displaystyle \ab(\xi^{\di n})=\ab(\xi)^n$ for any $\xi\in \Hon(K,G)$ and $n\in \Z$.

\item[\rm(v)] $1^{\di n} = 1$ and $(\xi^{\di m})^{\di n} = \xi^{\di mn}$ for every $m, n \in \Z$.

\end{enumerate}
\end{proposition}

Our proof of Proposition~\ref{p:ab-n} will be based on the following lemma.

\begin{lemma}\label{t:Bor98-T}
For any reductive group $G$ over a local or global field $K$,
and for any cohomology class $\xi\in \Hon(K,G)$,
there exists a maximal torus $i\colon T\into G$ over $K$ and a cohomology class $\xi_T\in \Hon(K,T)$
such that $\xi=i_*(\xi_T)$.
\end{lemma}

\begin{proof}
This follows from \cite[Lemma 10.2]{Kottwitz-86}
in the local field case (including the case $K=\R$),
from \cite[Theorem 5.10]{Borovoi-Memoir} in the number field case,
and from \cite[Corollary 1.10]{Thang} in the function field case.
\end{proof}

\begin{proof}[Proof of Proposition~\ref{p:ab-n}]
By Lemma~\ref{t:Bor98-T} there exists a maximal torus $i\colon T\into G$
defined over $K$, and a cohomology class  $\xi_T\in\Hon(K,T)$ such that $\xi=i_*(\xi_T)$.
Since $\Hon(K, T)$ is an abelian group, the identity we want to prove in each part
is valid for $\xi_T$; see Example~\ref{ex:torus}. Now apply $i_*$ and use Proposition~\ref{d6}
to deduce the desired identity for $\xi$.

For example, to prove (iii), we observe that
since $A$ is central in $G'$, we have $A\subseteq T'$
where $T'$ is the preimage of $T$ in $G'$.
Consider the commutative diagram
\begin{equation*}
\begin{aligned}
\xymatrix{
1 \ar[r] &A\ar[r]\ar@{=}[d] &T'\ar[r]\ar[d]^-{i'} &T\ar[r]\ar[d]^-i &1\\
1 \ar[r] &A\ar[r]           &G'\ar[r]             &G\ar[r]              &1
}
\end{aligned}
\end{equation*}
whose rows are central short exact sequences of algebraic groups.
Moreover, the algebraic groups in the top row are {\em abelian}.
This diagram induces a commutative diagram of Galois cohomology sets
\begin{equation*}
\begin{aligned}
\xymatrix@C=10mm{
\Hon(K,T)\ar[r]^{\delta_T}\ar[d]_-{i_*} & \Ho^2(K,A)\ar@{=}[d]\\
\Hon(K,G)\ar[r]^\delta           & \Ho^2(K,A)
}
\end{aligned}
\end{equation*}
in which the connecting map $\delta_T$ is a homomorphism of abelian groups.
Here the commutativity of the diagram means that
\begin{equation}\label{e:com}
\delta_T=\delta\circ i_*\hs.
\end{equation}
We wish to compute $$\delta(x^{\di n})=\delta\big(i_*(\xi_T)^{\di n}\big).$$
By Proposition \ref{d6}  (functoriality in $G$), we have
\[ i_*(\xi_T)^{\di n}=i_*(\xi_T^{\di n})=i_*(\xi_T^n),\]
whence
\[ \delta(x^{\di n}) = \delta\big(i_*(\xi_T^n)\big).\]
By \eqref{e:com} we have
\[\delta\big(i_*(\xi_T^n)\big)=\delta_T(\xi_T^n).\]
Since the map $\delta_T$ is a  group homomorphism, we have
\[ \delta_T(\xi_T^n)=\delta_T(\xi_T)^n.\]
By \eqref{e:com} we have
\[\delta_T(\xi_T)=\delta\big(i_*(\xi_T)\big)=\delta(\xi),\]
and therefore,
\[ \delta_T(\xi_T)^n=\delta(\xi)^n.\]
Thus
\[\delta(x^{\di n})=
\delta\big(i_*(\xi_T)^{\di n}\big)=
\delta\big(i_*(\xi_T^n)\big)=
\delta_T(\xi_T^n)=
\delta_T(\xi_T)^n=
\delta(\xi)^n,
\]
as desired.
\end{proof}

\begin{example} \label{ex.PGLn} Consider the short exact sequence
$1 \to \GG_m \to \GL_n \to \PGL_n \to 1$
defined over a field $K$.
It is well known that
\begin{itemize}
\item the elements of $\Hon(K, \PGL_n)$ are in a natural bijective
correspondence with the set of isomorphism classes of central simple algebras of degree $n$ over $K$; see~\cite[p.~396]{involutions};

\item
the second cohomology group $\Ho^2(K, \GG_m)$ is naturally isomorphic to the Brauer group $\Br(K)$; see~\cite[Theorem 4.4.3]{gille-szamuely};
and

\item the connecting map $\delta \colon \Hon(K, \PGL_n) \to \Ho^2(K, \GG_m)$ taking
a central simple algebra $A$ to its Brauer class $[A]$ is injective;
see \cite[Lemma 2.4.4 and Proposition 2.7.9]{gille-szamuely}
or the paragraph following formula (29.10)  on p.~396 in~\cite{involutions}.
\end{itemize}
For $G = \PGL_n$, Proposition~\ref{p:ab-n}(iii)
tells us that for any  global  field $K$, our power operation $\di$ on $\Hon(K, \PGL_n)$
is induced by the power operation in the Brauer group $\Br(K)$.
That is, if $\xi \in \Hon(K, \PGL_n)$ is represented by the central simple $K$-algebra $A$ of degree $n$,
then $\xi^{\di d}$ is represented by the unique (up to isomorphism)
central simple $K$-algebra $B$ of degree $n$
that is Brauer-equivalent to
$A^{\otimes d}$.
\end{example}

We conclude this section by showing that the power operation $\di$
of Construction~\ref{con:power} is, in fact, unique.

\begin{proposition} \label{p:uniqueness}
Let $K$ be a local or global  field and $n$ be an integer.
Suppose that for any reductive group $G$ over $K$ we have a power map
$p_{n,G} \colon \Hon(K, G) \to \Hon(K, G)$ with the following two properties:
\begin{enumerate}
\item[\rm (i)] $p_{n,G}$ is functorial in $G$, and\smallskip
\item[\rm (ii)] It $T$ is a torus defined over $K$, then
$p_{n,T}(\xi) = \xi^n$ for every $\xi \in \Hon(K,T)$.
Here $\xi^n$ denotes the $n$-th power of $\xi$ in the abelian group $\Hon(K, T)$.
\end{enumerate}
Then $p_{n,G}(\xi) = \xi^{\di n}$ for every $\xi \in \Hon(K, G)$.
\end{proposition}

In particular, the functoriality in $K$ (Proposition \ref{p:ab-n}(i)) is a formal consequence of (i) and (ii).

\begin{proof} Let $\xi \in \Hon(K, G)$. By Lemma \ref{t:Bor98-T} there exists
a maximal torus $i\colon T \into G$ defined over $K$ and a cohomology class
$\xi_T \in \Hon (K, T)$ such that $\xi=i_*(\xi_T)$.
By (ii) we have $p_{n,T}(\xi_T) = \xi_T^n = \xi_T^{\di n}$.
Now (i) tells us that
\[
p_{n,G}(\xi)=p_{n,G}\big(i_*(\xi_T)\big) =i_*\big(p_{n,T}(\xi_T)\big)
  =i_*( \xi_T^{n})=\big(i_*(\xi_T)\big)^{\di n}=\xi^{\di n}. \qedhere\]
\end{proof}

We conclude this section with a proof of Theorem \ref{t:non-ex},
which asserts that it is impossible to define
a power map as in Question \ref{q:main} over an arbitrary field.
As we shall see, Theorem \ref{t:non-ex} is an easy consequence
of Proposition~\ref{p:Gille} due to Philippe Gille.

\begin{proof}[Deduction of Theorem \ref{t:non-ex} from Proposition \ref{p:Gille}]
Assume the contrary: there exists a functorial power map $\di n$
satisfying condition (i) of Question~\ref{q:main}. Let $G$ be a reductive $K$-group,
$\iota\colon T\into G$ be a $K$-torus in $G$ and
$\iota_*\colon \Hon(K,T)\to \Hon(K,G)$ be the induced map.
As we remarked after the statement of
Question~\ref{q:main}, $1_\sG^{\di n} = 1_\sG$ for any integer $n$,
where  $1_\sG\in \Hon(K,G)$ denotes the neutral cohomology class.
This tells us that if $\xi\in \Hon(K,T)$ is such that  $\iota_*(\xi)=1_\sG$, then
\[ \iota_*(\xi^n)=1_\sG^{\di n}= 1_\sG\hs. \]
On the other hand, by Proposition~\ref{p:Gille}, there exist
$K$, $G$, $\iota\colon T\into G$, and $\xi\in \Hon(K,T)$
such that $\iota_*(\xi) = 1_\sG$ but $\iota_*(\xi^2) \neq 1_\sG$.
This contradiction completes the proof of Theorem \ref{t:non-ex}.
\end{proof}

\section{An explicit description of the power map}
\label{s:explicit}

In Construction \ref{con:power}, when defining the power map ${\di n}$,
we used the description of the pointed set $\Hon(K,G)$ as the fiber product
of $\Ho^1_\ab(K,G)$ and $\Hon(K_\infty,G)$.
Let $T\subset G$ be a maximal torus.
Then
\[\Ho^1_\ab(K,G)\coloneqq\HH^1(K,T^\ssc\to T)=
      \HH^1\big(\G(K^s/K), \big(\X_*(T^\ssc)\to \X_*(T)\big)\otimes (K^s)^\times\big).\]
In this section we give a simpler description of the pointed set $\Hon(K,G)$,
without mentioning the $\G(K^s/K)$-module $(K^s)^\times$; see Theorem \ref{t:BK} below.
This leads to a description of the power map for $\Hon(K, G)$ in Proposition~\ref{t:explicit}
that is more concrete  and more amenable to computations than our original definition.
We shall use this description  in Section~\ref{sect.examples}.

Let $\G$ be a group and let $A$ be a $\G$-module,
that is, an abelian group on which $\G$ acts by $(\ga,a)\mapsto \ga\cdot a$.
Consider the functor
\[A\,\functor\, A_\Gt\hs.\]
Here $A_\G$ is the {\em group of coinvariants of $\G$ in $A$}, that is,
\[A_\G\coloneqq A/\langle \ga\cdot a-a\ |\ a\in A,\, \ga\in\G\rangle,\]
and $A_\Gt\coloneqq (A_\G)_\tors$, the torsion subgroup of $A_\G$.

Let $G$ be a reductive group over a local or  global  field $K$.
In each case, there is an explicit description of the pointed set $\Hon(K,G)$.
Let $d\in\Z$ be an integer.
Below we give an explicit description of the power map
\begin{equation}\label{e:di-d}
\di d\colon \Hon(K,G)\to\Hon(K,G),\quad \xi \, \mapsto \, \xi^{\di d}
\end{equation}
in each case.

When $K=\R$,  there is an explicit description of $\Hon(\R,G)$ in \cite[Theorem 7.14]{BT-IJM}.
However, we do not need it here. The map  $\di d$ of \eqref{e:di-d} is the identity map if $d$ is odd,
and sends all elements to 1 if $d$ is even.

In the other cases, consider the $\G(K^s/K)$-module $M\coloneqq\pi_1^\alg(G)$.
Choose a finite Galois extension $L/K$ in $K^s$ such that
$\G(K^s/L)$ acts on $M$ trivially.
Write $\G=\G(L/K)$, which is a finite group naturally acting on $M$.

When $K$ is a non-archimedean local field,  we have a canonical bijection
\begin{equation} \label{e.bk6.5.2}
\Hon(K,G)\isoto M_\Gt\hs  ;
\end{equation}
see~\cite[Theorem 5.5.2(1)]{BK}.

\begin{proposition}\label{p:explicit}
When $K$ is a non-archimedean local field,
the following diagram is commutative:
\[
\xymatrix@C=14mm{
\Hon(K,G)\ar[r]^-{\di d} \ar[d]_-\sim  &\Hon(K,G)\ar[d]^-\sim\\
M_\Gt\ar[r]^-{\bdot d}             &M_\Gt
}
\]
where \,$\bdot d\colon M_\Gt\to M_\Gt$ \,is the endomorphism of multiplication by $d$.
\end{proposition}

\begin{proof}
We have a diagram
\[
\xymatrix@C=14mm{
\Hon(K,G)\ar[r]^-{\di d} \ar[d]^\ab_-\sim     &\Hon(K,G)\ar[d]^-\ab_-\sim\\
\Hon_\ab(K,G)\ar[r]^-{\bdot d} \ar@{<-}[d]^-\phi_-\sim  &\Hon_\ab(K,G)\ar@{<-}[d]^-\phi_-\sim\\
M_\Gt\ar[r]^-{\bdot d}             &M_\Gt
}
\]
where $\phi$ is the isomorphism of \cite[Theorem 5.5.2(1)]{BK}.
Here the top rectangle commutes by the definition of the power map $\di d$
over a non-archimedean local field,
and the bottom rectangle commutes because $\phi$ is a group isomorphism.
Thus the diagram of the proposition commutes, which completes the proof.
\end{proof}

Now let $K$ be a  global  field (a number field or a function field), and let
$M$, $L/K$, and $\G$ be as above.
Let $\V_L$ denote the set of all places of $L$.
Consider the group of {\em finite} formal linear combinations
\[ M[\V_L]=\bigg\{\sum_{w\in \V_L} m_w\cdot w\ \big |\  m_w\in M\bigg\},\]
(here $m_w=0$ for almost all $w$),  and the subgroup
\[ M[\V_L]_0=\bigg\{\sum_{w\in \V_L} m_w\cdot w\in M[\V_L]
   \ \ \,  \Big | \ \sum_{w\in \V_L} m_w=0\bigg\}. \]
The finite group $\G=\G(L/K)$ acts on $M$ and on $\V_L$,
and so it naturally acts on $M[\V_L]$ and on $M[\V_L]_0$.
Consider the functor
\[G\ \functor\ M\ \functor \big(M[\V_L]_0\big)_\Gt\hs.\]
This functor does not depend on the choice of $L$; see \cite[Lemma 4.3.1]{BK}.
For any $v\in\V_\infty(K)$, choose a place $\bv$ of $L$ over $v$
and let $\G_{\bv}$ be the corresponding decomposition group,
that is, the stabilizer of $\bv$ in $\G$.
Now consider the functor
\[G\ \functor\ M\ \functor\!\!\! \bigoplus_{v\in\V_\infty(K)}\!\!\! M_\Gbvt\hs.\]
We have natural maps
\begin{equation}\label{e:H1-BK}
\xymatrix@1{
\big(M[\V_L]_0\big)_\Gt\ \ar[r]^-{l_\infty}
     &\!\!\displaystyle{\bigoplus\limits_{v\in\V_\infty(K)}\!\!\! M_\Gbvt}\ \,
     &\ \Hon(K_\infty\hs, G),\ar[l]_(.4){\theta_\infty}
}
\end{equation}
defined as follows. For each $v\in \V_\infty(K)$, we consider the ``localization'' homomorphism
\[ l_v\colon \big(M[\V_L]_0\big)_\Gt\to\,  M_\Gbvt\, ,
\quad \ \Bigg[\,\sum_{w\in\V_L} m_w\cdot w\,\Bigg]\ \mapsto \,[m_\bv]\hs;\]
then the homomorphism $l_\infty$ in \eqref{e:H1-BK}
is the product of the homomorphisms $l_v$ over $v\in \V_\infty(K)$.
Moreover, for each $v\in \V_\infty(K)$, we consider the map
\begin{equation}\label{e:Ab-Ho}
\theta_v\colon\Hon(K_v,G)\labelto{\ab} \Ho^1_\ab(K_v,G)\overset\sim\longleftrightarrow
    \Ho^{-1}(\G\hm_\bv,M)\hs\into\hs M_\Gbvt\hs,
\end{equation}
where $\Ho^{-1}(\G\hm_\bv,M)\labelto\sim \Ho^1_\ab(K_v,G)$ is the Tate-Nakayama isomorphism;
see \cite[Proposition 8.21]{BT-IJM}.
See~\cite[Chapter IV, Section 6]{CF} for the definition of the Tate cohomology groups $ \Ho^{-1}$.
The rightmost arrow in \eqref{e:Ab-Ho} is a special case of the inclusion of
Tate (-1)-cohomology
$$\Ho^{-1}(\Delta,A)\into A_{\Delta,\Tors}\into A_\Delta$$
for a group $\Delta$ and a $\Delta$-module $A$.
Then the map $\theta_\infty$ in \eqref{e:H1-BK} is
the product of the maps $\theta_v$ over $v\in \V_\infty(K)$.

\begin{definition}
The pointed set $\cH^1(G/K)$ is the fiber product of the two maps in \eqref{e:H1-BK}.
\end{definition}

We need the following result:

\begin{theorem}[\hs{\cite[Theorem 5.5.10]{BK}}\hs]
\label{t:BK}
For a reductive group $G$  over a  global  field $K$,
with algebraic fundamental group $M=\pi_1^\alg(G)$,
there is a canonical bijection
$\Hon(K,G)\isoto\cH^1(G/K)$,
functorial in $G$ and $K$.
\end{theorem}

To specify the bijection of the theorem, we need to construct two maps:
\[\phi_M\colon \Hon(K,G)\to \big(M[\V_L]_0\big)_\Gt\quad \text{and}
\quad\ \loc_\infty\colon \Hon(K,G)\to\Hon(K_\infty,G).\]
The map $\phi_M$ is the composition
\[\Hon(K,G)\labelto{\ab} \Ho^1_\ab(K,G)\labelto{\psi_M} \big(M[\V_L]_0\big)_\Gt\hs\]
where $\psi_M$ is the injective {\em homomorphism}
of \cite[Theorem 4.3.14(1) and Remark 4.3.8]{BK}.
The  (surjective) infinite localization map $\loc_\infty$
is the map given by formula \eqref{e:loc-inf} in Section \ref{s:infinite}.
We obtain a map
\[\phi_\cH=(\phi_M,\loc_\infty)\colon\, \Hon(K,G)\isoto\cH^1(G/K),\]
which is bijective by the theorem.

\begin{proposition}\label{t:explicit}
Let $G$ be a reductive group with algebraic fundamental group $M=\pi_1^\alg(G)$,
defined over a  global  field $K$, and let $d\in\Z$ be an integer.
Then the power map $$\di d\colon \Hon(K,G)\to\Hon(K,G)$$
fits into the following commutative diagram:
\[
\xymatrix@C=18mm{
\Hon(K,G)\ar[r]^-{\di d}\ar[d]^-{\phi_\cH}_-\sim & \Hon(K,G)\ar[d]^-{\phi_\cH}_-\sim\\
\cH^1(G/K)\ar[r]^-{(\hs\bdot d,\hs\dit d\hs)}            & \cH^1(G/K)
}
\]
where the self-map $(\hs\bdot d,\hs\dit d\hs)\colon \cH^1(G/K) \to \cH^1(G/K)$
is given by the maps
\[ \bdot d\colon \big(M[\V_L]_0\big)_\Gt\to \big(M[\V_L]_0\big)_\Gt\hs\quad\ \text{and}
\quad\ \dit d\colon \Hon(K_\infty,G) \to \Hon(K_\infty,G).\]
Here $\bdot d$ is the endomorphism of multiplication by $d$
in the abelian group $\big(M[\V_L]_0\big)_\Gt$\hs,
and $\dit d$ is the map of Section \ref{s:real}.
\end{proposition}

\begin{proof}
We need to show that the following two diagrams commute:
\begin{equation}\label{e:di-d-bdot-d}
\begin{aligned}
\xymatrix@C=12mm{
\Hon(K,G)\ar[r]^-{\di d}\ar[d]^-{\phi_M} & \Hon(K,G)\ar[d]^-{\phi_M}
& \Hon(K,G)\ar[r]^-{\di d}\ar[d]^-{\loc_\infty} & \Hon(K,G)\ar[d]^-{\loc_\infty}\\
 \big(M[\V_L]_0\big)_\Gt\ar[r]^-{\bdot d}            & \big(M[\V_L]_0\big)_\Gt
& \Hon(K_\infty,G)\ar[r]^-{\dit d}            & \Hon(K_\infty,G)
}
\end{aligned}
\end{equation}
The diagram at right commutes by the definition of the power map $\di d$.
To prove the commutativity of the diagram at left, we consider the following diagram:
\[
\xymatrix@C=13mm{
\Hon(K,G)\ar[r]^-{\di d}\ar[d]^-\ab & \Hon(K,G)\ar[d]^-\ab\\
\Ho^1_\ab(K,G) \ar[r]^-{x\hs\mapsto x^d}\ar[d]^-{\psi_M} & \Ho^1_\ab(K,G)\ar[d]^-{\psi_M} \\
\big(M[\V_L]_0\big)_\Gt\ar[r]^-{\bdot d}            & \big(M[\V_L]_0\big)_\Gt
}
\]
In this diagram, the top rectangle commutes by the definition of the power map $\di d$,
and the bottom rectangle commutes because $\psi_M$ is a group homomorphism.
This proves the commutativity of the diagram at left in \eqref{e:di-d-bdot-d},
and completes the proof of the proposition.
\end{proof}

\section{The period divides the index}
\label{s:period}

Let $G$ be a reductive group over a local or global field $K$ and let $\xi\in \Hon(K,G)$.
In this section we recall the definitions of the period and the index of $\xi$ from the Introduction
and show that the period divides the index.

\begin{definition} \label{def:index}
 Let $\xi\in \Hon(K,G)$ where  $G$ is a reductive group over an arbitrary field $K$.
The {\em index} $\ind(\xi)$ is the greatest common divisor
of the degrees $[L:K]$ as $L/K$ ranges over finite field extensions  that split $\xi$.
\end{definition}

\begin{remark} \label{rem.degree} One can also define the {\em separable  index} $\ind_\sep(\xi)$
as the greatest common divisor  of the degrees $[L: K]$
where $L/K$ ranges over all finite {\em separable} field extensions splitting $\xi$.
Note that if $X \to \Spec(K)$ is a torsor representing
$\xi \in H^1(K, G)$, then $\ind(\xi)$  (resp.,  $\ind_\sep(\xi)$\hs)
is the same as the index $\delta(X/K)$ of $X$ (resp., the separable index $\delta_\sep(X/K)$\hs)
in the sense of~\cite[Section 9]{GLL}.
Therefore, we have
\[ \ind(\xi)=\delta(X/K)=\delta_\sep(X/K)=\ind_\sep(\xi),\]
where the middle equality is a special case of \cite[Theorem 9.2]{GLL}
(in fact, \cite[Theorem 9.2]{GLL} establishes this equality for any regular,
generically smooth, non-empty scheme $X$ of finite type over $K$).
We thus conclude that $\ind(\xi)$ is the greatest common divisor  of the degrees $[L: K]$
of all finite {\em separable} field extensions $L/K$ splitting $\xi$.
\end{remark}

\begin{definition} \label{def:period}
Suppose $K$ is a local or  global  field.
By the {\em period} $\per(\xi)$ we mean the greatest common divisor
of all positive  integers $d$ such that $\xi^{\di d}=1\in \Hon(K,G)$.
\end{definition}

\begin{example} \label{ex.PGLn-b} Suppose $G = \PGL_n$ and $K$ is a  global  field.
Recall from Example~\ref{ex.PGLn} that elements of
$\Hon(K, \PGL_n)$ are in a natural bijective correspondence
 with isomorphism classes of central simple algebras
of degree $n$ over $K$. Denote by $A$ a central simple algebra associated to
$\xi \in \Hon(K, \PGL_n)$.
Then a field extension $L/K$ splits $\xi$ if and only if it splits $A$.
In particular, the index of $\ind(\xi)$
is the same as the index of $A$;
see~\cite[Proposition 4.5.1]{gille-szamuely}.
Moreover, by Example~\ref{ex.PGLn}, for any integer $d$ the class
$\xi^{\di d} \in \Hon(K, \PGL_n)$ corresponds to a unique (up to isomorphism) central simple $K$-algebra $B$
of degree $n$ that is Brauer-equivalent to $A^{\otimes d}$.
In particular, $\xi^{\di d} = 1$ in $\Hon(K, \PGL_n)$
if and only if $A^{\otimes d}$ is split over $K$.
We conclude that the period of $\xi$ is the same as the period of $A$,
or, equivalently, as the period of the class of $A$
in the Brauer group $\Br(K)$; see \cite[Definition 2.8.6]{gille-szamuely}.
\end{example}

\begin{proposition}\label{p:xi^per}
Let $\xi\in\Hon(K,G)$ where $G$ is a reductive group over a local or  global  field $K$.
Then
$\displaystyle \xi^{\di \per(\xi)}=1$.
\end{proposition}

\begin{proof}
First consider the case $K=\R$.
If $\per(\xi)=2$, then clearly we have $\xi^{\di \per(\xi)}=\xi^{\dit 2}=1$.
If $\per(\xi)=1$, then there exists an {\em odd} integer  $d>0$ such that $\xi^{\di d}=1$,
whence $\xi=1$ and $\xi^{\di \per(\xi)}=1$.

For a local or  global  field $K$ other than $\R$ and $\C$, write $\xi_\ab=\ab(\xi)\in\Ho^1_\ab(K,G)$.
Since $\per(\xi)$ is a linear combination of integers $n\in\Z$ such that $\xi^{\di n}=1$ (whence $(\xi_\ab)^n=1$),
we see that
\begin{equation}\label{e:ab-per}
 (\xi_\ab)^{\per(\xi)}=1.
\end{equation}
If $K$ is a non-archimedean local field of a  global  function field,
then from \eqref{e:ab-per}  it follows that $\xi^{\di \per(\xi)}=1$.

If $K$ is a number field, we write $\xi_\infty=\loc_\infty(\xi)\in \Hon(K_\infty, G)$.
If $\per(\xi)$ is even, then automatically we have $(\xi_\infty)^{\dit \per(\xi)}=1$,
and using \eqref{e:ab-per} we conclude that $\xi^{\di \per(\xi)}=1$.
If $\per(\xi)$ is odd, then there exists an {\em odd} integer $n$ such $\xi^{\di n}=1$,
whence $(\xi_\infty)^{\dit n}=1$ and $\xi_\infty=1$.
It follows that $(\xi_\infty)^{\dit \per(\xi)}=1$, and using \eqref{e:ab-per} we conclude that $\xi^{\di \per(\xi)}=1$.
\end{proof}

\begin{corollary}\label{c:xi^per}
For $\xi\in \Hon(K,G)$ where $G$ is a reductive group over a local or global field $K$,
the period $\per(\xi)$ is the least integer $n\ge 1$ such that $\xi^{\di n}=1$.
\end{corollary}

\begin{theorem}[\hs Theorem \ref{thm:divides}\hs]
\label{t:divides}
Let $\xi\in \Hon(K,G)$ where $G$ is a reductive group over a local or  global  field $K$.
Then $\per(\xi)\,|\, \ind(\xi)$.
In other words, if $\xi$ can be split by a finite separable extension of degree $n$, then $\xi^{\di n}=1$.
\end{theorem}

\begin{proof}
Assume that $\xi$ can be  split by a finite separable extension $L/K$ of degree $n$.

First consider the case $K=\R$.
If $n=2$, then automatically we have $\xi^{\di 2} =1$.
If $n=1$, then $\xi^{\di 1}=\xi=1$.

In the other cases,
since $\xi$ is split  by the extension $L/K$ of degree $n$,
we see that $\ab\,\xi$ is split by this extension, that is,
\begin{equation}\label{e:Res}
\Res_{L/K} (\ab\, \xi)=1.
\end{equation}
From the well-known formula
\[ \Cor_{L/K}\circ\Res_{L/K}=(x\mapsto x^n),\]
see~\cite[Chapter IV, Section 6, Proposition 8]{CF},
which is also valid for hypercohomology, see \cite[Remark 4.5.3]{BK},
we deduce from \eqref{e:Res} that
\begin{equation}\label{e:ab-xi-n}
(\ab\, \xi)^n=1.
\end{equation}

If $K$ is a non-archimedean local field or a  global  function field, then we deduce from \eqref{e:ab-xi-n}
that $\xi^{\di n}=1$.

In the case where $K$ is a number field, it remains to
show that
\begin{equation}\label{e:locinf-xi-n}
(\locinf\xi)^{\di n} =1.
\end{equation}

If $n$ is even, then \eqref{e:locinf-xi-n} holds automatically.
If $n$ is odd, then for each real place $v$ of $K$,
there exist at least  one {\em real}  place $w$ of $L$ over $v$.
Then $L_w=K_v$.
Since $\xi$ splits over $L$, it splits over $L_w$ and hence over $K_v$. Thus $\loc_v\hs\xi=1$.
We see that the assumption that $\xi$ is split by the extension $L/K$ of odd degree
implies that $\locinf \xi=1$, whence \eqref{e:locinf-xi-n} holds.

The equalities \eqref{e:ab-xi-n} and \eqref{e:locinf-xi-n} imply $\xi^{\di n}=1$.
This completes the proof of the theorem.
\end{proof}

\begin{corollary}[\hs{Sansuc~\cite[Corollary 4.8]{Sansuc}}\hs]
\label{c:Sansuc}
Let $G$ be a reductive group over a local or  global  field $K$,
and let $\{K_i\}$ be finite separable extensions of $K$
such that the greatest common divisor of their degrees $[K_i:K]$ is $1$.
Then the canonical map
\begin{equation}\label{e:Sansuc}
f\colon\Hon(K,G)\to \prod_i \Hon(K_i,G)
\end{equation}
is injective.
In particular, if a $G$-torsor $X$ over $K$ has a point in each $K_i$,
then it has a point in $K$.
\end{corollary}

\begin{proof}
By twisting, it suffices to show that the kernel of the map $f$ of \eqref{e:Sansuc} is trivial.
If $\xi\in\ker f$, then $\ind(\xi)=1$; see Definition~\ref{def:index}.
By Theorem~\ref{t:divides}
we have $\per(\xi)=1$, whence by Proposition \ref{p:xi^per} we obtain that
$\xi^{\di 1}=\xi^{\di\per(\xi)}=1$,  that is, $\xi=1$, as desired.
\end{proof}

\begin{remark} \label{p:same-R}
Let $\xi\in \Hon(\R,G)$, where $G$ is a reductive group over $\R$.
Then $\per(\xi)=\ind(\xi)$.

Indeed, $\per(\xi) = 1$ if $\xi = 1$ and $\per(\xi) = 2$ otherwise; see Section~\ref{s:real}.
If $\xi = 1$, then $\mathbb R$ splits $\xi$ and hence, $\ind(\xi) = 1 = \per(\xi)$.
If $\xi \neq 1$, then $\xi$ is not split by $\R$ but is split by $\C$, so $\ind(\xi) = 2 = \per(\xi)$.
\qed
\end{remark}

\begin{remark} \label{rem.totaro} Let $G$ be a reductive group over a field $K$ and $\xi \in \Hon(K, G)$.
Recall from Definition~\ref{def:index} that $\ind(\xi)$ is defined as the greatest common divisor of the degrees $[L:K]$
of finite separable field extensions $L/K$ splitting $\xi$ (or of all finite field extensions splitting $\xi$).
It is natural to ask if every $\xi$ can be split by a single separable field extension of degree $\ind(\xi)$.
Equivalently, if a $G$-torsor ${X} \to \Spec(K)$ has a zero cycle of degree $d$,
does it necessarily have a point of degree $d$?
In the case where $d = 1$, this question was posed by
Serre~\cite[p.~233]{Serre-PP}. When
$K$ is a local or  global  field, Corollary~\ref{c:Sansuc}
tells us that the answer is "Yes", whereas for a general field $K$, Serre's question is open.
For an arbitrary $d \geqslant 1$ this question is due to Totaro~\cite[Question 0.2]{totaro}. It
was recently answered in the negative by Gordon-Sarney and Suresh~\cite{suresh, suresh-erratum}.
In particular, Gordon-Sarney and Suresh constructed counter-examples
where $K = \Q_2$~\cite[Example 3.6]{suresh}
and where $K = \Q$~\cite[Example 3.7]{suresh}.
\end{remark}

\section{The period 2 property for semisimple groups}
\label{sect.examples}

\begin{definition}
Let $G$ be a semisimple
group defined over a local or  global   field $K$. We say that
$G$ has {\em the period $2$ property} if $\per(\xi) = 2$ for every $\xi\neq 1$ in  $\Hon(K, G)$.
Equivalently,
\begin{equation}\label{e:2-per}
\xi^{\di 2}=1\quad\ \text{for all}\ \, \xi\in \Hon(K,G)
\end{equation}
or
\begin{equation*}
\xi^{\di  n}=
 \begin{cases}
 \xi &\text{when $n$ is odd,}\\
 1 &\text{when $n$ is even.}
 \end{cases}
\end{equation*}
\end{definition}

\begin{proposition} \label{prop.sc}
Let $G$ be a semisimple group defined over a local or global  field $K$.
Let $M=\pi_1^\alg(G)$ be the algebraic fundamental group of $G$
and let $\mu\coloneqq\ker[G^\ssc\to G]$.
(Note that $\mu$ may be non-smooth.)

\smallskip

\begin{enumerate}
\item[\rm (a)] $2 \cdot M = 0$ if and only if $2\cdot\X^*(\mu)=0$.
Here $\X^*(\mu)$ denotes the group of characters of $\mu$ defined over $K^s$.
\smallskip

\item[\rm (b)] Assume that the equivalent conditions of part (a)
are satisfied.  Then $G$ has the period $2$ property.
\end{enumerate}
\end{proposition}

\begin{proof}
Part (a) follows from Proposition \ref{prop.ss}.
To prove part (b), assume that $2\cdot M=0$.
If $K=\R$, assertion \eqref{e:2-per} is obvious.
If $K$ is a non-archimedean local field,
then assertion \eqref{e:2-per} follows from $2\cdot M=0$
by Proposition \ref{p:explicit}.

If $K$ is a global field, consider the maps \eqref {e:H1-BK}  in
Section \ref{s:explicit}.
Since $2\cdot M=0$, we see that the map $\bdot 2$ on $\big(M[\V_L]_0\big)_\Gt$ is identically 0,
and clearly the map $\dit 2$ on $\Hon(K_\infty,G)$ is identically 1.
Thus the induced map  $(\bdot  2,\dit 2)$ on $\cH^1(G/K)$ is identically 1,
and by Proposition \ref{t:explicit} the map $\di 2$ on $\Hon(K,G)$
is identically 1. We conclude that $G$ has the period 2 property.
\end{proof}

\begin{example} \label{ex.sc}
The hypothesis $2\cdot \X^*(\mu)=0$ of Proposition \ref{prop.sc} is satisfied,
in particular, in the following cases:

\begin{itemize}
\item all simply connected semisimple $K$-groups,

\item
  $\SL(1,A)/\mu_2$\hs, where $A$ is a central simple algebra of even degree over $K$,

\item
 $\operatorname{SO}(q)$, where $q$ is a non-degenerate quadratic form of dimension $\geqslant 3$ over $K$,

\item
 adjoint simple $K$-groups of type $\CCC_n$,

\item
 adjoint simple $K$-groups of type $\EEE_7$,

\item
 adjoint simple $K$-groups of type $\DDD_{2m}$ for $m\ge 2$,

\item
half-spin group $\operatorname{HSpin}(q)$ where $q$ is a non-degenerate quadratic form
whose dimension is divisible by $4$ over $K$.
\end{itemize}

\smallskip \noindent
Therefore, all of these groups have the period $2$ property.
\end{example}

\begin{lemma}\label{l: perod2-local}
Let $K$ be a non-archimedean local field and $G$ be an {\emm inner form} of a split semisimple $K$-group.
Then $G$ has the period 2 property if and only if $2\cdot M=0$ where $M=\pi_1^\alg(G)$.
\end{lemma}

Recall that an {\em inner form} of a reductive $K$-group $G$
is a reductive group of the form $_cG$
where $c\in Z^1(K, G/Z_G)$ and $Z_G$ denotes the center of $G$.

\begin{proof}
The ``if" assertion follows from Proposition  \ref{prop.sc}.
To prove the ``only if" assertion,
assume that $2\cdot M\neq 0$. In other words, there exists $x\in M$ with $2x\neq 0$.
Now recall that $M$ is finite.
Moreover, since $G$ is an inner form of a split group,
the Galois group $\G=\Gal(K^s/K)$ acts trivially on $M$. Hence,
$M_\Gt \simeq M$, and we have a canonical bijection
\[ \Hon(K, G) \isoto M_\Gt \simeq M\]
by formula \eqref{e.bk6.5.2} in Section \ref{s:explicit}.
Let $\xi\in\Hon(K,G)$ be the element with image $x$ in $M$.
Since $2x\neq 0$, Proposition~\ref{p:explicit} tells us that $\xi^{\di 2}\neq 1$, as desired.
\end{proof}

\begin{lemma}\label{l:surjectivity}
Let $G$ be a {\emm semisimple} group over a global field $K$,
and let $v$ be a place of $K$.
Then the localization map at $v$
\begin{equation}\label{e:loc-BH}
\loc_v\colon \Hon(K,G)\to\Hon(K_v,G)
\end{equation}
is surjective.
\end{lemma}

\begin{proof}
When $K$ is a number field, this follows immediately from
\cite[Theorem 1.7]{BH}; see also \cite[Proposition 3.14]{Borovoi-CR23}.
The proof in \cite{Borovoi-CR23} uses \cite[Proposition 2.6]{Kottwitz-86}.

When $K$ is a global function field,
the same proof as in \cite{Borovoi-CR23},
but with \cite[Theorem 4.2]{Thang} instead of \cite[Proposition 2.6]{Kottwitz-86},
gives a proof of the surjectivity of \eqref{e:loc-BH} in this case.
\end{proof}

We are now ready to prove Theorem \ref{t:2-property-intro},
which we restate here for the reader's convenience.

\begin{theorem}[{Theorem \ref{t:2-property-intro}}]
\label{t:2-property}
Let $K$ be a  {\emm global}  field, and let $G$ be a semisimple $K$-group.
Then $G$ has the period 2 property if and only if $2\cdot \X^*(\mu)=0$
where $\mu=\ker [G^\ssc\to G]$.
\end{theorem}

\begin{proof}
The ``if" assertion follows from Proposition  \ref{prop.sc}. We prove``only if''.

Assume that $2\cdot M\neq 0$.
Let $\Theta=\Theta_K(M)$
denote the image of $\G(K^s/K)$ in $\Aut M$.
For a place $v$ of $K$, let $\Theta_v\subseteq\Theta$
denote the corresponding decomposition group (defined up to conjugacy).
By the Chebotarev density theorem there exists
a finite place $v$ of $K$ for which $\Theta_v=\{1\}$, and hence
$G_v\coloneqq G\times_K  K_v$ is an inner form of a split $K_v$-group.
We have $\pi_1^\alg(G_v) = \pi_1^\alg(G) = M$.
Since $2\cdot M\neq 0$,
by Lemma \ref{l: perod2-local} the group $G_v$ does not have the period 2 property,
that is, there exists $\xi_v\in \Hon(K_v,G_v)$ with $\xi_v^{\di 2}\neq 1$.
By Lemma \ref{l:surjectivity},  there exists $\xi\in \Hon(K,G)$ with $\loc_v(\xi)=\xi_v$\hs.
Then $\xi^{\di 2}\neq 1$, as desired.
\end{proof}

\begin{remark}\label{r:2-property}
Theorem \ref{t:2-property} fails over {\em non-archimedean local} fields!
Indeed, let $K$ be such a field, and let $L/K$ be a separable quadratic extension.
Take $G={\rm PU}_{3,L/K}$, the projective unitary group in 3 variables corresponding to the quadratic extension $L/K$.
For $M=\pi_1^\alg(G)$ we have $M\simeq \Z/3\Z$ with a {\em non-trivial} action of the Galois group $\G=\Gal(K^s/K)$
(because our $G$ is an {\emm outer} form of a split group).
It follows that $M_\G=0$ and $M_\Gt=0$.
By \cite[Theorem 5.5.2(1)]{BK} we have $\Hon(K,G)\cong M_\Gt=0$,
and therefore $G$ has the period 2 property.
On the other hand, $2\cdot M\neq 0$, and therefore $2\cdot\X^*(\mu)\neq 0$.
\end{remark}

\section{The transfer map}
\label{s:transfer}

Let $\G$ be a group\footnote{\,Later in this paper, $\G$ will usually denote a profinite group, but it is allowed to be an arbitrary abstract group in this section.}
and let $A$ be a $\G$-module.
Let $\Delta\subset \G$ be a subgroup of finite index.
Following Brown \cite[Section III.9(B)]{Brown}, we define the {\em transfer homomorphism}
\begin{equation}\label{e:NDG}
\Tsf_{\G/\Delta}\colon A_\Gt\to A_{\Delta,\Tors}
\end{equation}
as follows. (See also \cite[beginning of Section 2.7]{BK},
where our $\Tsf_{\G/\Delta}$ is denoted by $N_{\Delta\lmod \G}$\hs.)
Consider the set of right cosets $\Delta\bs \G$.
Choose a section $s\colon \Delta\bs \G\to \G$
of the natural projection $\G\to \Delta\bs \G$.
The homomorphism
\begin{equation*}
A\to A,\quad\ a\mapsto\!\! \sum_{x\in\Delta\lmod\G}\!\!{s(x)}\cdot a
\end{equation*}
descends to a homomorphism $A_\G\to A_\Delta$\hs,
which does not depend on the choice of the section $s$;
see Rotman \cite[Lemma 9.91(i)]{RotmanHA}.
Thus it induces a well-defined homomorphism \eqref{e:NDG}.

\begin{lemma}\label{l:mult-index}
Let $A$ be a $\G$-module, and let $\Delta\subset \G$ be a subgroup of finite index.
Suppose that the images of the natural homomorphisms
\begin{equation} \label{e:im-GD}
\G\to\Aut A\quad\ \text{and}\quad\ \Delta\to\Aut A
\end{equation}
coincide.
We identify
$A_{\Delta,\Tors}\isoto  A_\Gt$.
Then the transfer homomorphism
\[\Tsf_{\G/\Delta}\colon A_\Gt\to A_{\Delta,\Tors}=  A_\Gt\]
is given by multiplication by $ [\G:\Delta]$. That is,
$\Tsf_{\G/\Delta}(\alpha) = [\G:\Delta]\cdot \alpha$
for any $\alpha\in A_\Gt$.
\end{lemma}

\begin{proof}
Let $\alpha=[a]\in A_\Gt$ where $a\in A$. Let $x\in\Delta\!\lmod\!\G$, $g=s(x)\in\G$.
Since the images of the homomorphisms \eqref{e:im-GD} coincide, we have
\[g \cdot a=\,h\cdot a\quad\ \text{for some}\ \, h\in\Delta.\]
Thus
\[ g\cdot a=(\, g\cdot a-a)+a=(\, h\cdot a-a)+a\ \overset{\Delta}{\thicksim}\ a,\]
where $\overset{\Delta}{\thicksim}$ denotes the equivalence relation in $A$
corresponding to the projection $A\to A_\Delta$\hs.
We conclude that
\[\Tsf_{\G/\Delta}\hs[a]= \sum_{x\in\Delta\lmod\G}[s(x)\cdot a]=  \sum_{x\in\Delta\lmod\G} [a]=[\G:\Delta]\cdot [a],\]
as desired.
\end{proof}

\section{The index divides a power of the period: the local case}
\label{sect.local}

In this section we consider the index of a cohomology class $\xi\in \Hon(K,G)$
of period $n$ in the case when $K$ is a  local field.
We shall show that $\xi$ can be split by a cyclic extension
of degree dividing $n^r$ for some positive integer $r$.
From this and Theorem \ref{t:divides}, it follows that
the period and the index of $\xi$ have the same prime factors.
Our main result is Theorem \ref{t:n[a]-l}.

Let $G$ be a reductive group over a field $K$.
Write $M=\pi_1^\alg(G)$, and let $\Theta_K=\Theta_K(M)$
denote the image of $\G(K^s/K)$ in $\Aut M$.
Let $F\subset K^s$ be the subfield corresponding
to $\ker\big[\G(K^s/K)\to\Aut M\big]$; then $\G(F/K)\cong\Theta_K$.
Let $\vt =\#\Theta_K$ denote the order of the finite group $\Theta_K$.

Let $K_F^\ab\subseteq F$ denote the maximal abelian extension of $K$ in $F$.
It corresponds to the commutator subgroup
\[ [\Th_K,\Th_K]\subseteq\Th_K\cong\Gal(F/K).\]
Then $K_F^\ab/K$ is a Galois extension with Galois group
$\Th_K^\ab\coloneqq \Th_K/[\Th_K,\Th_K]$.
We set
\[\vt^\ab=\#\Th_K^\ab=[K_F^\ab:K];\]
then $\vt^\ab\,|\, \vt$.

\begin{definition}
Let $\vt \ge 1$ and $n>1$ be integers.
By the {\em $n$-capacity} $\cp(n,\vt)$ of  $\vt $
we denote the greatest divisor of  $\vt $
that divides $n^l$ for some positive integer $l$.
\end{definition}

Now assume that $K$ is a non-archimedean local field.

\begin{theorem}\label{t:n[a]-l}
Let $K$ be a non-archimedean local field
with absolute Galois group $\G_K=\G(K^s/K)$.
Let $G$ be a reductive $K$-group, and write $M=\pi_1^\alg(G)$.
Let $\xi\in \Hon(K,G)$, and assume that $\xi^{\di n}=1$
for some integer $n>1$.
Then any cyclic extension $L/K$ of degree divisible by
$n\cdot\cp(n,\vt^\ab)$ splits $\xi$.
\end{theorem}

\begin{remark} \label{rem.existence}
A non-archimedean local field $K$ admits an (unramified) cyclic extension $L/K$ of any given integer degree $m \geqslant 1$.  Indeed, the degree $m$ unramified cyclic extensions of $K$ are in bijective correspondence with the degree $m$ cyclic extensions of the residue field $k_K$; see~\cite[Section III.5]{Serre-LF}. Here $k_K$ is a finite field and hence, admits a cyclic extension of every degree.
\end{remark}

Before proceeding with the proof of Theorem~\ref{t:n[a]-l}, we record its immediate consequence.

\begin{corollary}\label{c:n[a]-l}
Let $K$, $G$ and $\xi$ be as in the statement of Theorem~\ref{t:n[a]-l}.
\begin{enumerate}

\item[\rm(a)] If $L/K$ is a cyclic field extension of degree $[L:K]$
divisible by $n\cdot\cp(n,\vt)$,
then $L/K$ splits $\xi$.
\smallskip

\item[\rm(b)]
Let $d =  \lfloor \log_2 \vt \rfloor$. If $L/K$ is a
cyclic extension of degree $n^l$ with $l\ge d+1$,
then $L/K$ splits $\xi$.
\end{enumerate}

\end{corollary}

\begin{proof}
(a) Since $\vt^\ab$ divides $\vt$, we conclude that
$ n\cdot\cp(n,\vt^\ab)$ divides $ n\cdot\cp(n,\vt)$.
Part (a) now follows from Theorem~\ref{t:n[a]-l}.

(b) It suffices to show  that $n^{d+1}$
is divisible by $n \cdot \cp(n, \vt)$, or equivalently,
$n^d$ is divisible by $\cp(n, \vt)$.
To prove this,
let \,$\cp(n,\vt)=\prod_i\hs p_i^{m_i}$ \,be the prime decomposition of $\cp(n,\vt)$.
Then $\prod_i 2^{m_i}\leqslant \cp(n,\vt)$, whence
\[2^{\sum m_i}\leqslant \cp(n,\vt)\quad\ \text{and}
    \quad\ \sum_i m_i\leqslant\log_2 \cp(n,\vt)\leqslant\log_2 \vt.\]

Write $k=\max m_i$.
For any prime factor $p_i$ of $\cp(n,\vt)$, clearly we have $p_i\hs|\hs n$,
whence
\begin{equation} \label{e.cp}
\cp(n,\vt)=\prod_i\hs p_i^{m_i}\ \big|\ \prod_i\hs p_i^{k}\ |\ n^k\ \ \text{with}\quad
      k\coloneqq \max_i m_i\leqslant\sum_i m_i\leqslant\lfloor\log_2 \vt\rfloor=d, \
\end{equation}
as desired.
\end{proof}

\begin{remark}
Theorem \ref{thm.main4} and Corollary \ref{c:main} in the non-archimedean local case
follow from Corollary \ref{c:n[a]-l}.
Indeed, by Corollary~\ref{c:n[a]-l},
a cyclic extension of degree $n^{d+1}$ with
$d=\lfloor\log_2 \vt\rfloor$ splits $\xi$, whence $\ind(\xi)\,|\,n^{d+1}$.
The archimedean (real) case is trivial;
see Remark~\ref{p:same-R}.
\end{remark}

Our proof of Theorem \ref{t:n[a]-l} will rely on the following lemma.

\begin{lemma}\label{l:n[a]}
Let $K$ be a non-archimedean local  field, $\G_K=\Gal(K^s/K)$ be the absolute Galois group of $K$,
and $M$ be a finitely generated $\G_K$-module.
Let $F\subset K^s$ be the finite Galois extension of $K$ corresponding
to the subgroup $\Delta \coloneqq \ker[\G_K \to\Aut M]$ of $\G_K$.
In other words, $(K^s)^\Delta = F$.

Let $n$ be a positive integer,
and $\alpha\in M_{\G_K,\Tors}$ be such that $n\hs\alpha=0$.
Let $L/K$ be a finite cyclic field extension in $K^s$ whose degree $[L:K]$
is divisible by $n\cdot \cp(n,\vt^\ab)$.
Then
\[ \Tsf_{\G_K/\G_L}(\alpha)=0 . \]
Here $\G_L=\Gal(K^s/L)$, and
$\Tsf_{\G_K/\G_L}\colon M_{\G_K,\Tors} \to M_{\G_L,\Tors}$
is the transfer homomorphism, as in Section~\ref{s:transfer}.
\end{lemma}

\begin{proof} First of all, we may assume without loss of generality that the degree $[L:K]$ of our cyclic extension $L/K$ divides $n^l$ for some $l \geqslant 1$.
Indeed, there is always an intermediate extension $L'/K$ in $L$ of degree $[L':K]=n\cdot\cp(n,\vt^\ab)$,
and $[L':K]$ divides some power of $n$. If we know that
$\Tsf_{\G_K/\G_{L'}}(\alpha)=0$, then we conclude that
we have
\[ \Tsf_{\G_K/\G_L}(\alpha) =  \Tsf_{\G_{L'}/\G_{L}}\big( \Tsf_{\G_K/\G_{L'}}(\alpha)\big)
=\Tsf_{\G_{L'}/\G_{L}} (0) = 0. \]
In other words, if Lemma~\ref{l:n[a]} is valid for $L'/K$, then it is also valid for $L/K$. After replacing $L$ by $L'$,  we may and shall assume that $[L:K]$ divides $n^l$ for some $l \geqslant 1$.

Now set $K_1 := L\cap F$.

\smallskip
{\bf Claim 1:} $[L: K_1]$ is divisible by $n$.

\smallskip
Since $[L:K] = [L:K_1] \cdot [K_1: K]$,
we only need to show that $[K_1:K]$ divides $\cp(n, \vt^\ab)$. For this, it suffices to prove that
(i) $[K_1:K]$ divides $\vt^\ab$ and (ii) $[K_1:K]$ divides some power of $n$.

Since the extension $L/K$ is cyclic, so is $K_1/K$,
and therefore $K_1\subseteq K_F^\ab$ and
\[ [K_1:K]\ \big|\ [K_F^\ab:K]=\vt^\ab.\]
This proves (i).
On the other hand,
$[K_1:K]\ \big |\ [L:K]\ |\ n^l$.
This proves (ii)  and thus completes the proof of Claim 1.

We now continue with the proof of Lemma~\ref{l:n[a]}.
The field $L$ is a cyclic extension of $K$, and hence a cyclic extension of $K_1$.
Therefore, there exists a unique subextension $K_1 \subset K_n \subset L$
such that $[K_n: K_1]=n$, as shown in the diagram below.
\[ \xymatrix @C=2pc @R=2pc
{ F \ar@{<-}[ddr]
&  L \ar@{<-}[d]  \\
                & K_n \ar@{<-}[d]^-{\text{degree  $n$}}  \\
                    & K_1 \ar@{<-}[d]  \\
                    & K \ar@/_5pc/[uuu]_{\rm cyclic} }\]

Set  $\G_1 := \Gal(K^s/K_1)$ and  $\G_n=\Gal(K^s/K_n)$. Since $[K_n:K_1] = n$,
$\G_n$ is a subgroup of $\G_1$ of index $n$.

\smallskip
{\bf Claim 2:} The images of $\G_n$ and $\G_1$ in $\Aut(M)$ coincide.

\smallskip
Claim 2 is equivalent to saying that $\G_n \cdot \Delta = \G_1\cdot\Delta$.
Since $F$ contains $K_1$, $\G_1$ contains $\Delta$. Thus we need to show that
\[ \G_n \cdot \Delta = \G_1. \]
By the Galois correspondence it suffices to check that
\begin{equation} \label{e.Galois}
(K^s)^{\G_n \cdot\hs \Delta} = (K^s)^{\G_1}.
\end{equation}
The left hand side of~\eqref{e.Galois} consists of elements of $K^s$ that are fixed by both $\G_n$ and $\Delta$.
The fixed field of $\G_n$ is $K_n$ and the fixed field of $\Delta$ is $F$. Thus
\[(K^s)^{\G_n \cdot\hs\Delta} = (K^s)^{\G_n} \cap (K^s)^\Delta = K_n \cap F  = K_1, \]
where the last equality follows from the inclusions
\[K_1\coloneqq L\cap F\supseteq K_n \cap F \supseteq K_1\cap F=K_1\hs.\]
On the other hand, since $\G_1=\Gal(K^s/K_1)$, the right hand side of~\eqref{e.Galois}
is also $K_1$.
This proves \eqref{e.Galois} and Claim 2.

\smallskip
It follows from Claim 2 that  $M_{\G_1} = M_{\G_n}$\hs.
By Lemma~\ref{l:mult-index},
the transfer map \[ \Tsf_{\G_1/\G_n} \colon M_{\G_1, \Tors} \to M_{\G_n, \Tors} \] is multiplication by $n$.
Now observe that the transfer map $\Tsf_{\G_K/ \G_L} \colon M_{\G_K, \Tors} \to M_{\G_L, \Tors}$
factors as the composition
\[ \Tsf_{\G_K/ \G_L} \; \colon \; M_{\G_K, \Tors} \xrightarrow{\Tsf_{\G_K/ \G_1}}
M_{\G_1, \Tors} \xrightarrow{\Tsf_{\G_1/ \G_n}}
M_{\G_n, \Tors} \xrightarrow{\Tsf_{\G_n/ \G_L}}
M_{\G_L, \Tors} . \]
Since $\alpha$ is $n$-torsion and the middle arrow is multiplication by $n$, we conclude that
$\Tsf_{\G_K/ \G_L}$ sends $\alpha$ to $0$.
This completes the proof of Lemma \ref{l:n[a]}.
\end{proof}

\begin{proof}[Proof of Theorem \ref{t:n[a]-l}]
Since $K$ is a non-archimedean local field,
we have a bijective morphism of pointed sets
\[\ab\colon\, \Hon(K,G)\,\isoto\, \Ho^1_\ab(K,G)\coloneqq\HH^1(F,T^\ssc\to T);\]
see \cite[Theorem 5.5.1(1)]{BK}. We denote
\[\xi_\ab=\ab(\xi)\in \HH^1(K,T^\ssc\to T).\]
Since $\xi^{\di n}=1$, by Proposition \ref{p:ab-n}(iv) we have $\xi_\ab^n=1$.
Recall that
\[M=\pi_1^\alg(G)=\coker\big[\X_*(T^\ssc)\to\X_*(T)\big].\]
Let $\alpha\in M_{\G_K,\Tors}$ denote the preimage of $\xi_\ab$ under the Tate-Nakayama  isomorphism
\[  M_{\G_K,\Tors}\isoto \HH^1(K,T^\ssc\to T) \]
of \cite[Theorem 4.2.7(1)]{BK}.
Since $\xi_\ab^n=1$, we have $n\hs\alpha=0$.
By Lemma \ref{l:n[a]} we have $\Tsf_{\G_K/\G_L}(\alpha)=0$.

By \cite[Proposition 4.5.1]{BK}, the following diagram commutes:
\begin{equation*}
\begin{aligned}
\xymatrix@C=12mm{
\HH^1(K,T^\ssc\to T)\ar[r]^-\sim\ar[d]_-{\Res_{L/K}}  & M_{\G_K,\Tors}\ar[d]^{\Tsf_{\G_K/\G_L}}\\
\HH^1(L,T^\ssc\to T)\ar[r]^-\sim                      & M_{\G_L,\Tors}
}
\end{aligned}
\end{equation*}
By Lemma~\ref{l:n[a]}, $\Tsf_{\G_K/\G_L}(\alpha)=0$.
We now see from the diagram that  $\Res_{L/K}(\xi_\ab)=1$,
whence by Proposition \ref{p:ab}(2),
it follows from the bijectivity of the abelianization map that
$\Res_{L/K}(\xi)=1$, as desired.
\end{proof}

\section{Period and index in the Sylow-cyclic case}
\label{sect.main2}

Let $G$ be a reductive group over a global field $K$.
Recall that $\Sha^1(K,G)$ denotes the Tate-Shafarevich kernel:
\[\Sha^1(K,G)=\ker\Big[\Hon(K,G)\to\prod_{v\in \V_K} \Hon(K_v,G)\Big]\]
where $\V_K $ denotes the set of places of $K$.

\begin{definition}
Let $K$ be a global field, and let $S\subset\V_K$ be a finite set of places of $K$.
Write $S=S_f\cup S_\R\cup S_\C$\hs,
where $S_f=S\cap\V_f(K)$, $S_\R=S\cap\V_\R(K)$, $S_\C=S\cap\V_\C(K)$
are the subsets of finite, real and complex places in $S$, respectively.
We say that a finite Galois extension $L/K$ of degree $m$
is {\em of full local degree in $S$}, if for each $v\in S_f$
we have $[L_{w}:K_{v}]=m$ for a place $w$ of $L$ over $v$,
that is, $L\otimes_K K_v$ is a field.
Moreover, if $m$ is even, we further require that for each  $v\in\ S_\R$,
all places $w$ of $L$ over $v$ should be complex.
\end{definition}

\begin{proposition}\label{p:AT}
Let $K$ be a global field, $S\subset \V_K$ be a finite set of places,
and $F/K$ be a finite Galois extension in $K^s$.
Let $n_L\ge 1$ be an integer.
Then there exists a cyclic extension $L/K$ in $K^s$ of degree $n_L$ of full local degree in $S$
and such that $L\cap F=K$.
\end{proposition}

\begin{proof}
By the Chebotarev density theorem, there exists a {\em finite} place $v_0$
such that $F\otimes_K K_{v_0}\simeq K_{v_0}^{[F:K]}$.
After extending $S$, we may and shall assume that $S$ contains $v_0$.
By \cite[Theorem X.6]{AT} there exists a cyclic extension $L/K$ in $K^s$
of degree $n_L$ of full local degree in $S$.
Write $E=F\cap L$ and $m_E=[E:K]$.
Since $E\otimes_K K_{v_0}$ is contained in the field $L\otimes_K K_{v_0}$,
we see that  $E\otimes_K K_{v_0}$ has no zero divisors,
and hence it is a field of degree $m_E$ over $K_{v_0}$\hs.
Since the field $E\otimes_K K_{v_0}$
is contained in the ring $F\otimes_K K_{v_0}\simeq K_{v_0}^{[F:K]}$,
we conclude that $m_E=1$, whence $F\cap L=K$.
\end{proof}

\begin{theorem}\label{t:Sha}
Let $G$ be a reductive group over a global field $K$, and let $\xi\in\Hon(K,G)$.
Assume that $\xi^{\di n}=1$ for some integer $n\ge 1$.
Let the finite group $\Theta_K=\Theta_K(M)$
and the finite Galois extension $F/K$ in $K^s$ with Galois group $\Theta_K$
be as in Section~\ref{sect.local}.
Then for any positive integer $m$ divisible by
$n\cdot\cp(n,\#\Theta)$,
there exists a cyclic field extension $L/K$
in $K^s$ of degree  $m$
such that $\Res_{L/K}(\xi)\in \Sha^1(L,G)$
and $L\cap F=K$.
\end{theorem}

\begin{proof}
Let $S(\xi)$ denote the set of places of $K$ for which $\xi_{v}\neq 1$
where $\xi_v=\loc_v(\xi)\in \Hon(K_v,G)$.
We write $S(\xi)=S_\R(\xi)\cup S_f(\xi)$, where $S_\R$ denotes
the set of real places and $S_f$ denotes the set of finite places.
The set $S(\xi)$ is finite;
see \cite[Proposition 6.6 on page 297]{PR} in the number field case, and
the case of a global function field is similar.

Let $v\in\V_\R(K)$. Write $\xi_v=\loc_v(\xi)$.
Since $\xi^{\di n}=1$, we have $\xi_v^{\dit n}=1$.
If $n$ is odd, then
\[\xi_v=\xi_v^{\dit n}=1,\quad\ \text{whence}\ \, v\notin S_\R(\xi).\]
We see that if $n$ is odd, then $S_\R(\xi)=\varnothing$.

By Proposition \ref{p:AT}
there exists a cyclic extension $L/K$ in $K^s$ of degree $m$
of full local degree in $S$ and such that  $F\cap L=K$.
For each $v\in S_f(\xi)$, the decomposition group $\Theta_{K_v}$
is a subgroup of $\Theta_{K}$ (defined up to conjugacy),
and therefore we have
\[ n\cdot\cp(n,\#\Theta_{K_v})\ \big|\ n\cdot\cp(n,\#\Theta_{K})\ \big|\ m,\]
whence by Corollary  \ref{c:n[a]-l} the cyclic extension $L_{w}/K_{v}$ of degree $m$
splits $\xi_{v}$\hs.
If $S_\R(\xi)$ is non-empty, then $n$ and $m$ are even, and for all $v\in S_\R(\xi)$,
the extension $L_{w}/K_{v}\simeq \C/\R$
splits $\xi_{v}$\hs.
This means that $\Res_{L/K}(\xi)\in \Sha^1(L,G)$, as desired.
\end{proof}

\begin{corollary}\label{c:Sha}
Let $r\ge 1+ \lfloor\log_2\vt\rfloor$ where  $\vt=\#\Theta_K$.
Then Theorem \ref{t:Sha} holds with $m=n^r$.
\end{corollary}

\begin{proof}
Indeed, then $n\cdot\cp(n,\vt)\,|\, n^r$; see formula~\eqref{e.cp} in Section~\ref{sect.local}.
\end{proof}

Following Andersen \cite{Andersen}, we say that a finite group is {\em Sylow-cyclic} if all its Sylow subgroups are cyclic
(Sansuc~\cite[p.~13]{Sansuc} calls such groups {\em metacyclic}).

\begin{proposition}
\label{p:B11}
Let $G$ be a reductive group over a global field $K$
with algebraic fundamental group $M=\pi_1^\alg(G)$,
and assume that the finite group $\Theta_K\subseteq\Aut M$ is Sylow-cyclic.
Then for any finite separable extension $L$ of $K$ we have $\Sha^1(L,G)=1$.
\end{proposition}

\begin{proof}
The group $\Theta_L$ is a subgroup of $\Theta_K$ and hence is Sylow-cyclic as well.
Now the proposition follows from  \cite[Corollary 7.4.4(1)]{BK}.
\end{proof}

\begin{corollary} \label{cor.sylow-cyclic}
Let $G$ be a reductive group over a global field $K$,
and assume that $\Theta_K$ is Sylow-cyclic.
Let $\xi\in\Hon(K,G)$, and assume that $\xi^{\di n}=1$ for some positive integer $n$.
Then $\xi$ can be split by a cyclic extension $L/K$
of degree  $n\cdot \cp(n,\#\Theta_K)$.
\end{corollary}

\begin{proof}
Indeed, by Theorem \ref{t:Sha}
there exists a cyclic extension $L/K$  of degree $n\cdot\cp(n,\#\Theta_K)$
such that
\[\Res_{L/K}(\xi)\in\Sha^1(L,G),\]
and by Proposition \ref{p:B11} we have  $\Sha^1(L,G)=\{1\}$.
Thus $\Res_{L/K}(\xi)=1$, that is, $L/K$ splits $\xi$, as desired.
\end{proof}

\begin{definition}
Let $G$ be a reductive group over a field $K$
with algebraic fundamental group $M=\pi_1^\alg(G)$,
and let $\Theta_K$ denote the image of the absolute Galois group $\G_K$ in $\Aut M$.
We say that  $\pi_1^{\alg}(G)$ is {\em split over} $K$ if $\Theta_K=\{1\}$,
that is, if $\G_K$ acts on $M$ trivially.
\end{definition}

\begin{remark}\label{r:split}
If  $G$ is split, then  $\pi_1^\alg(G)$ is split as well.
Indeed, choose a split maximal torus $T\subseteq G$;
then $\Gal(K^s/K)$ acts trivially on $\X_*(T)$,
and hence on $\pi_1^\alg(G)=\X_*(T)/\X_*(T^\ssc)$.
However, the assertion that  $\pi_1^\alg(G)$ is split over $K$ does not imply that $G$ is split.
For instance, take $G=\SO_n$ over $\R$ with $n\ge 3$.
Then $\pi_1^\alg(G)\cong\Z/2\Z$, hence split, whereas $G$ is anisotropic, hence not split.
\end{remark}

We are now ready to proceed with the proof of Theorem~\ref{thm.main2},
which we restate here for the reader's convenience.

\begin{theorem}[Theorem \ref{thm.main2}]
\label{t:split}
Let $G$ be a reductive group over a local or global field $K$ for which
the algebraic fundamental group $M=\pi_1^\alg(G)$ is split over $K$.
Then for any $\xi\in \Hon(K,G)$ we have $\per(\xi)=\ind(\xi)$.
\end{theorem}

\begin{proof} Let $K$ be a global field.
Suppose $\per(\xi) = n$.
By Theorem  \ref{thm:divides} we have $n=\per(\xi) \,|\, \ind(\xi)$.
On the other hand, since $\Theta_K=\{1\}$,
it is Sylow-cyclic, and
Corollary~\ref{cor.sylow-cyclic} tells us that $\xi$ can be split
by a cyclic extension $L/K$
of degree $n\cdot\cp(n,\#\Theta_K)=n\cdot 1 =n$.
This shows that $\ind(\xi) \, | \, n$. We conclude that
$\ind(\xi) = n = \per(\xi)$, which proves the theorem in the global case.
The local case is similar: we use Corollary \ref{c:n[a]-l}(a) instead of Corollary~\ref{cor.sylow-cyclic}.
\end{proof}

\section{The index divides a power of the period: the global case}
\label{sect.main4-0}

In this section we prove Theorem~\ref{thm.main4} in the case
where $K$ is a global field.

\begin{theorem}
\label{t:ind-per-d}
Let  $\xi \in \Hon(K, G)$ where $G$ is a reductive group over a global field $K$.
Then  $\ind(\xi)$ divides $\per(\xi)^d$ for some positive integer $d$.
\end{theorem}

From this theorem and Theorem \ref{t:divides}, it follows that  the period and the index of $\xi$ have the same prime factors.

As a first step, we establish a version of Theorem~\ref{t:ind-per-d} for a torus.
Here $K$ is allowed to be an arbitrary field.

\begin{proposition} \label{lem.main1b-1}
Let $T$ be a torus of dimension $d$ over a field $K$, and let  $\xi \in \Hon(K, T)$.
Assume that $\xi^n = 1$ for some $n>0$. Then $\xi$ can be split
by a finite field extension $L/K$ whose degree $[L:K]$ divides $n^d$.
\end{proposition}

\begin{proof} The exact sequence
$\xymatrix{ 1\ar[r] & T[n] \ar[r] & T \ar[r]^{\times\, n} & T \ar[r] & 1 }$
of algebraic groups induces an exact sequence
\[ \xymatrix @C=1pc @R=.5pc{ \Hon(K, T[n]) \ar[rr] & & \Hon(K,T) \ar[rr]^{\times\,  n} & & \Hon(K, T) \\
                   \alpha\, \ar@{|->}[rr] & & \, \xi\,  \ar@{|->}[rr] & & \, 1} \]
in flat cohomology. Here $T[n]$ is the $n$-torsion subgroup of $T$ \footnote{\,If $\Char(K)$ does not divide $n$, then $T[n]$ is smooth, and Galois cohomology coincides with flat cohomology. In general $T[n]$ may not be smooth and we need to use flat cohomology to get an exact sequence.}.
  Since $\xi^n = 1$, the above sequence shows that $\xi$ is the image of some $\alpha \in \Hon(K, T[n])$. It remains to show that $\alpha$ can be split by a field extension of degree dividing $n^d$. Since every class in $\Hon(K, T[n])$ is represented by a $T[n]$-torsor over $K$, it suffices to show that
 \begin{equation} \label{e.main1b-1}
 \begin{aligned}
 &\text{Every $T[n]$-torsor $\pi \colon X \to \Spec(K)$
 can be split}\\ &\text{\qquad\qquad\qquad by a field
 extension $L/K$ of degree dividing $n^d$.}
 \end{aligned}
 \end{equation}

\smallskip
Write $n = p^r \cdot n_0$, where $p = \text{char-exp}(K)$ is the characteristic exponent of $K$,
and $n_0$ is prime to $p$. If $p = 1$, then
$p^r = 1$ and $n_0 = n$. Then $T[n] \simeq T[p^r] \times T[n_0]$. (If $a$ and $b$ are integers such that $ap^r + b n_0 = 1$, then this isomorphism is given explicitly by $t \to (t^{bn_0}, \, t^{ap^r})$.)
Here $T[p^r]$ is an infinitesimal group over $K$ and
$T[n_0]$ is a smooth finite group of order $n_0^d$ over $K$
(which is a $K$-form of $(\Z/ n_0\Z)^d$). Since $\Hon(K, T[n]) =
\Hon(K, T[p^r]) \times \Hon(K, T[n_0])$, it suffices to prove~\eqref{e.main1b-1} in two special
cases, where (1) $n = n_0$ is prime to $p$ and (2) $n = p^r$.

\medskip
{\bf Case 1:} $n$ is prime to $p$.
In this case $T[n]$ is an \'etale (or equivalently, smooth) finite  $K$-group
(of order $n^d$). Hence, $X$ is also smooth over $K$.

\smallskip
{\bf Claim:} Let $H$ be a smooth  finite algebraic group of order $m$ over a field $K$.
 Then every $H$-torsor $\pi \colon X \to \Spec(K)$ can
 be split by a finite field extension $L/K$
 whose degree $[L:K]$ divides $m$.

\smallskip
Over the algebraic closure $\overline{K}$ of $K$, $X$ becomes isomorphic to $H$.
In particular, $X_{\overline{K}}$ is affine over $\overline{K}$.
By~\cite[Proposition 2.7.1(xiii)]{EGA4}, this implies that $X$ is affine over $K$.
Let us denote the coordinate ring of this variety by $A = K[X]$.
Then $\dim_K(A) = m$ is finite. Consequently,
every element $x \in A$ is a root of a monic polynomial with coefficients in $K$.
In other words, the ring $A$ is integral over $K$.

Now assume for a moment that $X$ is irreducible over $K$. Then $A$ is an integral domain.
Since $K \subset A$ is an integral extension of integral domains, and $K$ is a field,~\cite[Proposition 5.7]{AM}
tells us that $A$ is also a field.
In this case we can take $L = A$.
Indeed, $X = \Spec(A) = \Spec(L)$, and so $X$ has an $L$-point. In other words, $L$ splits $\pi$.
Here $[L:K] = \dim_K(A) = m$.

Now consider the general case where $X$ is allowed to have multiple irreducible components
$X_1, \ldots, X_t$ over $K$. In this case $H$ transitively permutes $X_1, \ldots, X_t$.
Let $H_1$ be the stabilizer of $X_1$ in $H$
and let $x_1\in X_1(K^s)$; then
\[ H_1(K^s)=\big\{ h \in H(K^s)\ \big|\ h \cdot x_1\in X_1(K^s)\big\}.\]
We see that $H_1(K^s)$ acts transitively on $X_1(K^s)$, and so $X_1$ is an $H_1$-torsor.
As we saw above, $X_1 = \Spec(A_1)$ where $A_1$ is a field containing $K$.
Taking $L$ be the $A_1$, we see that $X_1$, and hence $X$, have an $L$-point;
thus $L$ splits $\pi$.
On the other hand,
\[ [L:K] =\, \#H_1\,|\ \#H = m.\]
This completes the proof of the claim.

In Case 1, Proposition \ref{lem.main1b-1}  follows from the claim applied
to the smooth finite group $H=T[n]$ of order $m=n^d$.
\medskip

{\bf Case 2:} $\text{char-exp}(K) = p > 1$ and $n = p^r$.
Once again, by~\cite[Proposition 2.7.1(xiii)]{EGA4}, $X = \Spec(A)$ is an affine variety over $K$, and $\dim_K(A) = \# (T[n]) = n^d$. Here
$A$ denotes the coordinate ring of $X$, as in Case 1.

Since $T[n]$ is an infinitesimal group, over any perfect field $K'$ containing $K$ we have $H^1(K', T[n]) = 1$.
(This follows, for instance, from~\cite[Lemma 11.1(b)]{reichstein-rd} with $G = T[n]$ and $G_{\rm red} = 1$.)
Letting $K' = K^{\rm pc}$ be the perfect closure of $K$
(that is, the largest purely inseparable extension of $K$ inside $\overline{K}$),
we conclude that
$X$ has a $K^{\rm pc}$-point or, equivalently, there exists a homomorphism of $A$-algebras
$\eta \colon A \to K^{\rm pc}$.
Let $\mathcal{M} \subset A$ be the
kernel of $\eta$ and $L = A/\mathcal{M}$ be the residue field of $\eta$.
By our construction, $L$ is an intermediate field between $K$
and its perfect closure $K^{\rm pc}$.
This tells us that $L$ is a purely inseparable extension of $K$ and consequently,
$[L:K]$ is a power of $p$. On the other hand,
$[L:K] = \dim_K(L) \leqslant \dim_K(A) = n^d = p^{rd}$.
We conclude that $[L:K]$ divides $p^{rd} = n^d$. This completes the proof of Proposition \ref{lem.main1b-1}.
\end{proof}

\begin{lemma}  {\rm
(\cite[Proof of Lemma 3.2]{MSh})}
\label{l:res-mng}
Let $\G$ be a finite group and $M$ be a finitely generated $\G$-module.
Let $\mg(M)$ denote the minimal number of generators
of the finitely generated abelian group $M$.
Then there exists a resolution
\begin{equation}\label{e:M-M0-M-1}
0\to M^{-1}\labelto{\vk} M^0\labelto{\lambda} M\to 0
\end{equation}
where $M^{-1}$ is a finitely generated  free $\Z[\G]$-module
(a finitely generated free module over the ring $\Z[\G]$\hs),
and $M^0$ is a $\Z$-torsion-free $\G$-module of $\Z$-rank at most $\,\#\G\cdot\mg(M)$.
\end{lemma}

Let $G$ be a reductive group over a global field $K$.
Set $M=\pi_1^\alg(G)$, and let $\Theta_K(M)$ (or just $\Theta_K$)
denote the image of $\G(K^s/K)$ in $\Aut M$.
Write $\G=\Theta_K$, $\vt=\#\Theta_K$.
Consider a short exact sequence of $\G$-modules \eqref{e:M-M0-M-1},
where $M^{-1}$ and $ M^0$ are finitely generated   $\Z$-torsion-free $\G$-modules,
$M^{-1}$ is $\Z[\G]$-free, and the $\Z$-rank of $M^0$ is at most $\vt\cdot\mg(M)$.
Write $T^i$ for the $K$-torus with $\X_*(T^i)=M^i$ for $i=-1,0$.
Then $T^{-1}$ is a {\em quasi-trivial} $K$-torus
(which means that $\X_*(T^{-1})$ admits a $\G(K^s/K)$-stable basis),
and $\dim T^0\le \vt\cdot\mg(M)$.
Using the notations of Section~\ref{s:abelian},
we have
\begin{gather*}
\coker[\X_*(T^\ssc)\to \X_*(T)]=M=\coker[M^{-1}\to M^0],\\
\ker[\X_*(T^\ssc)\to \X_*(T)]=0\quad\ \text{and}\quad\ \ker[M^{-1}\to M^0]=0.
\end{gather*}
Therefore, we have a canonical isomorphism
\begin{align*}
\Ho^1_\ab(K,G)\coloneqq&\HH^1\big(\G_K, (\X_*(T^\ssc)\to \X_*(T)\hs)\otimes (K^s)^\times\big)\\
\cong&\HH^1\big(\G_K, (M^{-1}\to M^0)\otimes (K^s)^\times\big)
    =\HH^1(K,T^{-1}\to T^0);
    \end{align*}
see \cite[the text after Definition 2.6.2]{Borovoi-Memoir}.

Moreover, since the torus $T^{-1}$ is quasi-trivial, we have
\begin{align}
&\Hon(K',T^{-1})=1\ \text{for any extension $K'$ of $K$,} \label{e:qt-H1}\\
&\!\Sha^2(K,T^{-1})=1;\label{e:qt-Sha2}
\end{align}
see Sansuc \cite[Lemma 1.9]{Sansuc}.
\footnote{\,Sansuc established formula~\eqref{e:qt-Sha2}
under the assumption that $K$ is a number field,
but his proof remains valid in
the case of a global function field.}

Consider the short exact sequence of complexes
\[ 1\to\,(1\to T^0)\,\to\, (T^{-1}\to T^0)\,\to\,(T^{-1}\to 1)\,\to 1\]
and the induced hypercohomology exact sequence
\begin{equation}\label{e:T-1-T0}
1=\Hon(K,T^{-1})\labelto{\vk_*} \Hon(K,T^0)\labelto{\lambda_*} \HH^1(K,T^{-1}\to T^0)\labelto\delta \Ho^2(K,T^{-1})
\end{equation}
where $\Hon(K,T^{-1})=1$ by \eqref{e:qt-H1}.

\begin{lemma}\label{l:xi1ab-xi0}
Let $n\ge 1$.
With the above  assumptions and notations, let $$\xi\in \Sha^1(K,T^{-1}\to T^0)$$
be such that $\xi^n=1$.
Then $\xi=\lambda_*(\xi_0)$ for a unique element $\xi_0\in \Hon(K,T^0)$.
Moreover, $\xi_0^n=1$.
\end{lemma}

\begin{proof}
Since $\xi\in \Sha^1(K,T^{-1}\to T^0)$, we have $\delta(\xi)\in \Sha^2(K,T^{-1})$,
and by \eqref{e:qt-Sha2} we have $\delta(\xi)=1$. By \eqref{e:T-1-T0},
$\xi=\lambda_*(\xi_0)$ for a unique element $\xi_0\in \Hon(K,T^0)$.
Using~\eqref{e:T-1-T0} one more time,
we see that $\xi^n=1$ implies $\xi_0^n\in\ker\lambda_*=\{1\}$.
\end{proof}

We are now ready to finish the proof of Theorem \ref{t:ind-per-d}.
Let $K$ be a global  field.
Let  $\xi\in \Hon(K,G)$ and let $n>0$ be such that $\xi^{\di n}=1$.
Our goal is to show that $\Res_{L/K}(\xi)=1$
for some finite  extension $L/K$ of degree $n^r$ for suitable $r$.
This will prove that
$\ind(\xi)$ divides $n^r$.

By Theorem~\ref{t:Sha} there exists
a cyclic extension $L_1/K$  of degree $n^{r_1}$ for some
integer $r_1 \geqslant 0$ such that $\xi_1\coloneqq\Res_{L_1/K}(\xi)$ lies in $\Sha^1(L_1,G)$,
and by Corollary~\ref{c:Sha} we can take $r_1=\lfloor \log_2\vt\rfloor +1$.
Since the power map $\di n$ is functorial,  we have
\[ \xi_1^{\di n}=(\Res_{L_1/K}\hs\xi)^{\di n}
    =\Res_{L_1/K}(\xi^{\di n})=\Res_{L_1/K}(1)=1.\]

Set $\xi_{1,\ab}=\ab(\xi_1)\in \Ho^1_\ab(L_1,G)$.
When $K$ is a number field, by \cite[Theorem 5.12]{Borovoi-Memoir},
for any finite  extension $L/L_1$  the natural map
\[
\ab\hm_\Sha\colon \Sha^1(L,G)\ \lra\, \Sha^1_\ab(L,G)\coloneqq
\ker\Big[\Ho^1_\ab(L,G)\to \prod\Ho^1_\ab(L_w,G)\Big]
\]
is bijective.
When $K$ is a global function field, the map $\ab\hm_\Sha$ is clearly bijective,
because then the corresponding abelianization maps
\[\ab\colon \Hon(K,G)\to \Ho^1_\ab(K,G)\quad\ \text{and}
     \quad\ \ab_v\colon \Hon(K_v,G)\to \Ho^1_\ab(K_v,G)\]
are bijective.
It follows that in both cases a finite extension $L/L_1$ that splits $\xi_{1,\ab}$\hs, also splits $\xi_1$.
By Proposition \ref{p:ab-n}(iv)  we have
\[(\xi_{1,\ab})^n\coloneqq(\ab\,\xi_1)^n=\ab(\xi_1^{\di n})=\ab(1)=1.\]

Write $\G_1=\Theta_{L_1}(M)$.
We choose a resolution \eqref{e:M-M0-M-1}
such that $M^{-1}$ is a finitely generated  $\Z[\G_1]$-free module.
Consider the hypercohomology exact sequence \eqref{e:T-1-T0}
where
\[ \HH^1(L_1,T^{-1}\to T^0)=\Ho^1_\ab(L_1,G).\]
Here $\dim T^0\le \vt\cdot\mg(M)$.
By Lemma \ref{l:xi1ab-xi0}, we have $\xi_{1,\ab}=\lambda_*(\xi_0)$
for some $\xi_0\in\Hon(L_1,T^0)$ such that $\xi_0^n=1$.

By Proposition~\ref{lem.main1b-1} there exists a finite extension $L/L_1$
of degree dividing $n^{\dim T^0}$ such that $L/L_1$ splits $\xi_0$.
Hence, $L/L_1$ splits $\xi_{1,\ab}$ and $\xi_1$ as well.
We obtain a tower of finite extensions
\begin{equation*}
\begin{aligned}
 \xymatrix@R=7mm
{   L \ar@{-}[d]^{\text{degree divides $n^{r_0}$}}  \\
                 L_1 \ar@{-}[d]^{\text{degree  $n^{r_1}$}}  \\
                     K }
\end{aligned}
\end{equation*}
where
\[ r_1=\lfloor \log_2\vt\rfloor+1 \quad \text{and} \quad r_0=\dim T^0\le\vt\cdot \mg(M).\]
In summary, we have constructed a field extension $L/K$ of degree dividing
$n^d=n^{r_1 + r_0}$ that splits $\xi$, which completes the proof of Theorem \ref{t:ind-per-d}.
\hfill\qed

\begin{remark} \label{rem.upper-bound-on-d}
Our proof of Theorem \ref{t:ind-per-d} shows that
we may choose the exponent $d$ satisfying
\[d\le \mg(M)\cdot\vt+\lfloor\log_2 \vt\rfloor+1\]
where $\vt =\#\Theta_K$ and $\mg(M)$ is the minimal number of generators
of the finitely generated abelian group $M$.
\end{remark}

\newcommand{\ZZ}{\Z}
\renewcommand{\CC}{\C}
\renewcommand{\QQ}{\Q}
\newcommand{\simlgr}{\buildrel\sim\over\longrightarrow}
\newcommand{\simla}{\buildrel\sim\over\longleftarrow}
\renewcommand{\Res}{{\rm Res}}
\newcommand{\Rost}{{\rm Rost}}
\newcommand{\irr}{{\rm irr}}

\appendix

\section{Examples: the index can be greater than the period}
\label{sect.example-local}

\centerline{\em by Mikhail Borovoi}
\bigskip

In this appendix we prove Theorem~\ref{thm.main3}.

\subsection{Separable extensions of even degree of non-archimedean local fields}
\label{ssn:quadratic}
In this subsection,  we show that
any finite separable extension of even degree of non-archimedean local fields
of residue characteristic not 2 contains a quadratic subextension.

Let  $K$ be a non-archimedean local field. We denote by $O_K$ the ring of integers of $K$,
by $\pp_K$ the maximal ideal of $O_K$,
and by $k_K=O_K/\pp_K$ the residue field.
The characteristic of the finite field $k_K$ is called the residue characteristic of $K$.

Consider the surjective homomorphism
\[O_K^\times\to k_K^\times,\quad\ a\mapsto a+\pp_K\]
with kernel $U^{(1)}=1+\pp_K$.
It admits a canonical splitting
\[\Teich_K\colon k_K^\times\isoto \mu_{q-1}(K)\into  O_K^\times\]
called the Teichm\"uller character; see \cite[Proof of Proposition II.5.3]{Neukirch}.
Here $q$ denotes the order of the finite field $k_K$\hs, and
$\mu_{q-1}(K)$ denotes the group of roots of unity of degree $q-1$ in $K$.
When $\alpha\in k_K^\times$, we shall write $\tilde\alpha=\Teich_K(\alpha)$
and say that $\tilde\alpha\in O_K^\times$ is
the {\em Teichm\"uller representative} of $\alpha$.
We obtain a canonical isomorphism
\[k_K^\times \times U^{(1)}\isoto O_K^\times,\quad\ (\alpha,u)\mapsto \tilde\alpha\hs u.\]

Choose a uniformizer of $K$, that is, a generator $\pi_K$ of the maximal ideal $\pp_K$ of $O_K$.
We have a non-canonical isomorphism
\begin{equation}\label{e:iso-Neu}
\Z\times  k_K^\times\times U^{(1)}\isoto K^\times,\quad (n,\alpha,u)\mapsto \pi_K^n\hs\tilde\alpha\hs u
\ \ \text{for}\ \ n\in\Z,\,\alpha\in k_K^\times,\,u\in U^{(1)};
\end{equation}
see \cite[Proposition II.5.3]{Neukirch}.

\begin{cons} \label{const:quadratic}
Let $K$ be a non-archimedean local field of residue characteristic $p$.
In this construction we assume that $p>2$;
then any element of $U^{(1)}\coloneqq 1+\pp_K$ is a square.

Write $\#k_K=q=p^r$. Then $q$ is odd, and therefore $k_K^\times$ is a  cyclic group of {\em even} order $q-1$.
It follows that there exists a non-square element $\alpha\in k_K^\times$.
Write $\ve=\tilde\alpha\in O_K^\times$.

Consider the following three non-square elements of $K^\times$:
\[a_1=\ve,\quad a_2=\pi_K, \quad a_3=\ve\pi_K\hs.\]
It follows from \eqref{e:iso-Neu} that $K^\times/(K^\times)^2\simeq \Z/2\Z\times\Z/2\Z$,
namely,
\[K^\times/(K^\times)^2=\big\{1,[a_1],[a_2],[a_3]\big\}\]
where $[a_j]$ denotes the class of $a_j\in K^\times$ in $K^\times/(K^\times)^2$.
Hence any separable quadratic extension of $K$ is isomorphic to $K(\sqrt{a_j}\hs)$ for some $j=1,2,3$.
\end{cons}

\begin{prop}[Spice \cite{LSpice}]
\label{p:LSpice}
Let $L/K$ be a finite separable  extension of {\emm even} degree of non-archimedean local fields
with residue characteristic $p>2$.
Then $L$ contains a quadratic subextension $F$ of $K$.
\end{prop}

Let $L/K$ be a finite separable extension of non-archimedean local fields.
Recall that $L/K$ is called {\em unramified} if $[L:K]=[k_L:k_K]$
where $k_L$ and $k_K$ are the corresponding residue class fields.
Then any uniformizer $\pi_K$ of $K$ is a uniformizer of $L$.
In general (not assuming that $L/K$ is unramified) we write $K^\ur_L$
for the unique maximal unramified subextension of $K$ in $L$;
for the proof of uniqueness see~\cite[Chapter I, Section 7, Theorem 2]{CF}.

If $L/K$ is an unramified  finite separable extension of non-archimedean local fields,
then $k_K$ naturally embeds into $k_L$,
and for any element $\alpha\in k_K^\times\subset k_L^\times$ we have
\[ \Teich_L(\alpha)=\Teich_K(\alpha)\in O_K^\times\subset O_L^\times\hs\]
where $\Teich_K\colon k_K^\times\to O_K^\times$ and $\Teich_L\colon k_L^\times\to O_L^\times$
are the corresponding Teichm\"uller characters;
see the beginning of this subsection.
We denote $\tilde\alpha=\Teich_L(\alpha)=\Teich_K(\alpha)\in O_K^\times$.

Recall that a finite separable extension $L/K$  of non-archimedean local fields
of residue characteristic  $p$
is called {\em (at most) tamely ramified} (or {\em tame}) if
the degree $[L:K^\ur_L]$ is prime to $p$.
In general (not assuming that $L/K$ is tame)
we write $K^\tame$ for the unique maximal tame subextension of $K$ in $L$;
for the proof of uniqueness see~\cite[Chapter I, Section 8, Theorem 1(i)]{CF}.
Then $[L:K^\tame]=p^r$ for some $r\ge 0$ ({\em loc.~cit.}).

\begin{proof}[Proof of Proposition \ref{p:LSpice}]
We consider the maximal tame subextension $L'=K^\tame$ of $K$ in $L$.
Since $p>2$, the degree $[L:L']=p^r$ is odd, whence $[L':K]$ is even.
After replacing $L$ by $L'$, we may and shall assume that our extension $L/K$ is tame.
We have a tower of extensions
\[ K\subseteq E\subseteq L\]
where $E\coloneqq K^\ur_L$ is the maximal unramified extension of $K$ in $L$.
Then $k_L=k_E$.
If  $f \coloneqq[E:K]$ is even, then we are done, because then $E/K$ is a cyclic extension of even degree.
We may thus assume that $[E:K]$ is odd and thus $[L:E]=2d$ for some
positive integer $d$.
Let $\pi_L$ be a uniformizer of $L$ and $\pi_K$ be a uniformizer of $K$.
Then $\pi_L^{2d}=a\pi_K$ for some $a\in O_L^\times$.
Let $\alpha$ denote the image of $a$ in $k_L^\times=k_E^\times$,
and let $\tilde\alpha\in O_E^\times$
denote the Teichm\"uller representative of $\alpha$; see the beginning of this subsection.
Then $\tilde\alpha^{-1}a\in U^{(1)}_L\coloneqq 1+\pp_L$.
Since $p$ is odd, we know that $\tilde\alpha^{-1}a=b^2$ for some  $b\in U^{(1)}_L$.
Thus $\pi_L^{2d}=\tilde\alpha\hs b^2\pi_K$.

\begin{lem}\label{l:Abrashkin}
The natural embedding $k_K^\times\into  k_E^\times$ induces  an isomorphism
\begin{equation}\label{e:K-E}
k_K^\times/(k_K^\times)^2\to k_E^\times/(k_E^\times)^2.
\end{equation}
\end{lem}

\begin{proof}[Proof of Lemma~\ref{l:Abrashkin}]
If the homomorphism \eqref{e:K-E} is not injective,
then there is a non-square element $\beta\in k_K$ that becomes a square in $k_E$,
and therefore the degree $[k_E:k_K]=f$ is even, a contradiction.
Thus \eqref{e:K-E} is injective.
Since each of the  groups $k_K^\times/(k_K^\times)^2$ and
$k_E^\times/(k_E^\times)^2$ is of order 2, we conclude that \eqref{e:K-E} is an isomorphism.
\end{proof}

We are now ready to complete the proof
of Proposition~\ref{p:LSpice}.
Lemma~\ref{l:Abrashkin} tells us
that $\alpha\in  k_E^\times$ can be written as $\alpha=\alpha_1^2\beta$
for some $\alpha_1\in k_E^\times$ and $\beta\in k_K^\times$,
whence
$$\pi_L^{2d}=\tilde\alpha_1^2\tilde\beta\hs b^2\pi_K=(\tilde\alpha_1 b)^2\tilde\beta\hs\pi_K$$
where $\tilde\alpha_1 \in O_E^\times $ is the Teichm\"uller representative of $\alpha_1$,
and where  $\tilde\beta\in O_K^\times$ is the Teich\-m\"uller representative of $\beta$.
Dividing both sides by $(\tilde{\alpha}_1 b)^2$, we obtain
\[(\pi_L^d b^{-1}\tilde\alpha_1^{-1})^2=\tilde \beta\pi_K\in K^\times.\]
The uniformizer $\tilde \beta\pi_K$ of $K$ is clearly not a square in $K$.
Now $K(\pi_L^d b^{-1}\tilde\alpha_1^{-1})$
is a desired quadratic subextension of $K$ in $L$.
This completes the proof of the proposition.
\end{proof}

\subsection{The index can be greater than the period: the local case}
\label{ssn:local}
We begin by restating  Theorem~\ref{thm.main3}(a) for the reader’s convenience.

\begin{thm}[Theorem~\ref{thm.main3}(a)]
\label{thm.main3-a}
Let $K$ be a non-archimedean local field of residue characteristic not $2$
containing $\zeta\coloneqq \sqrt{-1}$.
Then there exist a $6$-dimensional $K$-torus $T$
and a cohomology class $\xi\in \Hon(K,T)$ such that
$ \per(\xi)=2$, but $4\hs\hs|\ind(\xi)$.
\end{thm}

We shall break up the proof into a series of constructions and lemmas.

\begin{cons}[due to Tyler Lawson \cite{Lawson}]
\label{Step1}
Let $\Gamma=\Z/4\Z$. Consider the $\Gamma$-module
$M=\Z[\ii]$ where $\ii=\sqrt{-1}$   and the generator $\gamma=1+4\Z\in\Gamma$
acts on $M$ by multiplication by $\ii$.
Consider the norm map
$$N_\G\colon M_\G\to M^\G$$
where $M^\G$ is the group of invariants of $\G$ in $M$, and $M_\G$ is the group of coinvariants;
see~\cite[Chapter IV, Section 6]{CF}.
In our case, the norm map $N_\G$ is induced by the map
\[1+\gamma+\gamma^2+\gamma^3\colon M\to M.\]

We shall now consider the Tate cohomology groups $\Ho^n(\G,M)$
of the finite group $\G$; see \cite[Chapter IV, Section 6]{CF}.
The groups $\Ho^n(\G,M)$ are defined for all $n\in \Z$.
The case where $n = -1$ will be of particular interest to us;
recall that by definition $\Ho^{-1}(\G,M)=\ker N_\G$.
Since
\[1+\gamma+\gamma^2+\gamma^3=1+\ii+(-1)+(-\ii)=0,\]
we have
\begin{equation} \label{e.M_Gamma}
 \Ho^{-1}(\G,M)=\ker N_\G=M_\G=\Z[\ii]/(1-\gamma)\hs\Z[\ii]=\Z[\ii]/(1-\ii)\cong\Z/2\Z.
 \end{equation}

Let $\Delta=\langle \gamma^2\rangle\subset\Gamma$ denote the subgroup of index 2.
Similarly, then the norm map
$$N_\Delta\colon M_\Delta\to M^\Delta$$
is induced by multiplication by $1+\gamma^2=1+(-1)=0$, and therefore
\[\Ho^{-1}(\Delta,M)\coloneqq\ker N_\Delta=M_\Delta=\Z[\ii]/(1-\gamma^2)\hs\Z[\ii]=\Z[\ii]/(2).\]
The restriction  homomorphism
\[\Res_{\Gamma/\Delta}\colon\
\Z[\ii]/(1-\ii)= \Ho^{-1}(\Gamma,M)=M_\G\ \lra\  M_\Delta=\Ho^{-1}(\Delta,M) =\Z[\ii]/(2)\]
is given by the transfer map
\[\Tsf_{\G/\Delta}\colon\, [z]\mapsto [(1+\gamma) z]=[(1+\ii)z],\quad z\in\Z[\ii],\ \,  \gamma=\ii\]
of Section \ref{s:transfer}.
By Lemma \ref{l:(1+i)} below, this transfer map is injective.
Therefore, for $x\neq 0\in \Ho^{-1}(\Gamma,M)$ we have
\[2x=0,\quad\ \text{but}\quad\ \Res_{\Gamma/\Delta}(x)\neq 0.\]
\end{cons}

\begin{lem}\label{l:(1+i)}
The multiplication by $1+\ii$
\[1+\ii\colon \ \Z[\ii]/(1-\ii)\to \Z[\ii]/(2)\]
is injective.
\end{lem}

\begin{proof}
The group $\Z[\ii]/(1-\ii)\cong\Z/2\Z$ is of order 2 with nonzero element $[1]$.
The image of $[1]$ in $\Z[\ii]/(2)$ is
\[[1\cdot(1+\ii)]=[1+\ii]\neq 0 \in \Z[\ii]/(2),\]
as desired.
\end{proof}

\begin{cons}\label{Step2-local}
Let $K$ be a non-archimedean local field as in Theorem \ref{thm.main3-a}.
Let $K'/K$ be a separable quadratic extension.
We may write $K'=K(\sqrt{a})$ for some $a\in K^\times\smallsetminus K^{\times\hs2}$.
We set $F=K(b)$ with $b=\sqrt[4]{a}$.
Since $K\ni\zeta\coloneqq\sqrt{-1}$, the extension $F/K$ is a Kummer extension;
see~\cite[Chapter III, Section 2]{CF}.
In particular, it is a cyclic Galois extension
with Galois group $\G\cong \Z/4\Z$, where the generator $\gamma=1+4\Z\in\G$ sends $b$ to $\zeta b$.
Then $F$ contains $K'=K(b^2)$, and the subfield $K'$ corresponds
to the subgroup $\Delta=2\Z/4\Z$ of $\G$.

Let $T$ be a two-dimensional $K$-torus with cocharacter group $M=\Z[\ii]$
and with effective Galois group (acting on  $M$) $\Gal(F/K)=\Z/4\Z$
such that $\gamma=1+4\Z\in\Gamma$ acts on $M$ by multiplication by $\ii$.
Write $M^\vee=\Hom(M,\Z)$;
then $T=D(M^\vee)$ with the notation of Milne \cite[Theorem 12.23]{Milne},
and we have $T(K^s)\cong M\otimes K^{s\,\times}$.
Moreover,  we have
\[\Hon(K,T)\cong M_\Gt\hs;\]
see \cite[Theorem 4.2.7(1)]{BK}.
On the other hand, $M_\G \simeq \Z[\ii]/(1 - \ii) \simeq \Z/ 2 \Z$ is a group of order 2;
see \eqref{e.M_Gamma}. Thus $M_\Gt=M_\G$ and $\Hon(K,T) \cong \Z/ 2 \Z$.

Let $\xi$ denote the nontrivial element of $\Hon(K,T)$. Then $\xi^2=1$.
\end{cons}

\begin{lem}\label{l:Res-K'/K}
$\Res_{K'/K}(\xi)\neq 1$.
\end{lem}

\begin{proof}
 By \cite[Proposition 4.5.1]{BK} the restriction map
 $\Res_{K'/K}\colon \Hon(K,T)\to\Hon(K',T)$ fits into
 the commutative diagram below:
\begin{equation}\label{diad-5.5.1}
\begin{aligned}
\xymatrix{
	M_\G\ar[r]^-\sim \ar[d]_-{\Tsf_{\G/\Delta}}&\Hon(K, T)\ar[d]^{\Res_{K'/K}}\\
	M_{\Delta}\ar[r]^-\sim &\Hon(K',T)
}
\end{aligned}
\end{equation}
Here the transfer map $\Tsf_{\G/\Delta}$ is as in Section \ref{s:transfer}.
Denote the class of an element $m \in M = \Z[\ii]$ in $M_\G \simeq \Z/ 2\Z$ by $[m]$. Then
\[\Tsf_{\G/\Delta}[m]=[\hs(1+\ii\hs )m]\in M_\Delta\cong\Z[\ii]/(2).\]
We see that $\Tsf_{\G/\Delta}$ is multiplication by $1+\ii$, and by Lemma \ref{l:(1+i)}
the map $\Tsf_{\G/\Delta}$ is injective.
Since diagram \eqref{diad-5.5.1} commutes,
the vertical arrow $\Res_{K'/K}$
is injective as well. In other words,
for $1 \neq \xi\in  \Hon(K,T)$ we have
$ \Res_{K'/K}(\xi)\neq 1$.
\end{proof}

\begin{cons}\label{Step3-local}
Let  $K$ be as in Theorem \ref{thm.main3-a}.
Let $a_1,a_2,a_3\in K^\times$ be as in Construction \ref{const:quadratic}.
Then any quadratic extension of $K$ is isomorphic
to $K\big(\sqrt{a_j}\hs\big)$ for some $j=1,2,3$.

For each $j\in\{1,2,3\}$,
set $F_j=K\big(\sqrt[4]{a_j}\big)$ and $K_j'=K\big(\sqrt{a}\big)$.
Then $F_j/K$ is a Kummer extension with Galois group $\G=\Z/4\Z$
where $\gamma=1+4\Z\in \G$ sends
$\sqrt[4]{a}$
to $\zeta \sqrt[4]{a}$ with $\zeta=\sqrt{-1}$.
Moreover, $K_j'/K$ is a quadratic extension
corresponding to the subgroup $\Delta=2\Z/4\Z\subset\G$.

Let $T_j$ be the two-dimensional $K$-torus splitting over $F_j$
with cocharacter group $M=\Z[\ii]$
on which the generator $1+4\Z\in\Gamma=\G(F/K)$ acts by multiplication by $\ii$.
Let $\xi_j \in \Hon(K,T_j) \simeq \Z/ 2\Z$ be the nontrivial element.
Then $\xi_j^2=1$.
Now we set
\[T=T_1\times T_2\times T_3\hs,\quad\  \text{and} \quad \xi=(\hs\xi_1,\hs\xi_2,\hs\xi_3\hs)\in \Hon(K,T).\]
Clearly $\xi \neq 1$ but $\xi^2=1$. In other words,
$\per(\xi) = 2$.
\end{cons}

To finish the proof of Theorem \ref{thm.main3-a}, it remains to establish the following.

\begin{lem} \label{lem.4|index}
The index $\ind(\xi)$ is divisible by $4$.
\end{lem}

\begin{proof}
For the sake of contradiction,
assume the contrary: $\xi$ can be split by a field extension $L/K$ whose degree $[L:K]$
is finite but is not divisible by $4$.
Note that $[L:K]$ is divisible by $\ind(\xi)$, and Theorem~\ref{t:divides}
tells us that
$\ind(\xi)$ is divisible by $\per(\xi) = 2$. Hence, $[L:K] = 2d$, where $d$ is odd.

Since the residue characteristic of $K$ is not $2$,
by Proposition \ref{p:LSpice} there is an intermediate quadratic subextension
$K \subset K' \subset L$ such that $[K':K] = 2$ and
$[L: K'] = d$ is odd.
Then $K'= K'_j$ for some $j = 1, 2, 3$.
Here $K'_j := K(\sqrt{a_j})$, as in Construction~\ref{Step3-local}.
Let $\xi'_j  \coloneqq\Res_{K'/K}(\xi_j) \in  \Hon(K', T_j)$.

Now consider two extensions of $K'$:
$L/K'$ of degree $d$ and $F_j/K'$ of degree $2$. Here $F_j = K(\sqrt[4]{a_j})$, as in Construction~\ref{Step3-local}.
By our assumption, $L/K$ splits $\xi$ and thus $L/K'$ splits $\xi'_j$.
On the other hand, by the definition of the torus $T_j$, the extension $F_j/K$ splits $T_j$.
Hence, by Hilbert's Theorem 90, $\Hon(F_j, T_j) = \Hon(F_j, \mathds G_m^2) = 1$.
It follows that $F_j/K'$ splits $\xi'_j$.

In summary, both $F_j/K'$ and $L/K'$ split $\xi'_j$. Since
\[ \gcd\big(\hs[F_j:K'],\hs [L:K']\hs\big) = \gcd(2, d) = 1, \]
by Corollary \ref{c:Sansuc} we have $\xi'_j = 1$ in $\Hon(K', T_j)$,
which contradicts Lemma~\ref{l:Res-K'/K}.
This completes the proofs of Lemma \ref{lem.4|index}
and of Theorem \ref{thm.main3-a}.
\end{proof}

\subsection{The index can be greater than the period: the global case}
\label{ssn:global}
We begin by restating Theorem~\ref{thm.main3}(b) for the reader's convenience.

\begin{thm}[Theorem~\ref{thm.main3}(b)]
\label{thm.main3-b}
Let $K$ be a global field of characteristic not $2$ containing $\zeta\coloneqq\sqrt{-1}$.
Then there exist a $6$-dimensional $K$-torus $T$
and a cohomology class $\xi\in \Hon(K,T)$ such that
$ \per(\xi)=2$, but $4\hs\hs|\ind(\xi)$.
\end{thm}

Since $K$ is not of characteristic 2, there is a non-archimedean completion $K_v$ of $K$
with residue characteristic different from 2. We fix $v$ for the remainder of this appendix. We shall break up the proof of Theorem~\ref{thm.main3-b} into a series of constructions and lemmas. The overall idea is to choose $T$ and
$\xi$ so that they localize at $v$ to the $6$-dimensional $K_v$-torus $T_v$ together with the class $\xi_v \in \Hon(K_v, G_v)$
from Construction~\ref{Step2-local}.

\begin{cons}\label{Step3}
Let $\alpha_v\in k_v^\times$ be a non-square invertible element of the residue field $k_v$ of $K_v$,
and write $\ve_v=\tilde\alpha_v\in O_v^\times$ for its Teichm\"uller representative.
Set $a_{v,1}=\ve_v$\hs, and consider two uniformizers $a_{v,2}$ and $a_{v,3}=\ve_v a_{v,2}$  of $K_v$.
Set $K'_{v,j}=K_v\big(\sqrt{a_{v,j}}\hs\big)$ for $j=1,2,3$.
Then any separable quadratic extension of $K_v$ is isomorphic
to $K'_{v,j}$ for some $j=1,2,3$; see Construction \ref{const:quadratic}.
Set $a_{v,0}=1\in K_v$.
We have
\[  K_v^\times=\bigcup_{j=0}^3 a_{v,j}\cdot (K_v^\times)^2\hs.\]
Since the subgroup $(K_v^\times)^2\subset K_v^\times$ is open,
by weak approximation, see \cite[Chapter II, Section 15, Theorem]{CF},
for $j=1,2,3,$ we may choose
$a_j\in K^\times\cap  \big(a_{v,j}\cdot (K_v^\times)^2\big).$
We obtain three quadratic extensions
$K'_j=K\big(\sqrt{a_j}\hs\big)$ of $K$ for $j=1,2,3$,
such that $K'_j\otimes_K K_v\simeq K'_{v,j}$\hs.

We choose $j\in\{1,2,3\}$ and write $a=a_j$.
We set $F=K\big(\sqrt[4]{a}\big)$ and $K'=K(\sqrt{a})$.
Then $F/K$ is a Kummer extension with Galois group $\G=\Z/4\Z$,
where $\gamma=1+4\Z\in \G$ sends $b\coloneqq\sqrt[4]{a}$ to $\zeta b$.
Moreover, $K'/K$ is a quadratic extension
corresponding to the subgroup $\Delta=2\Z/4\Z\subset\G$.
We have
\begin{equation}\label{e:F-K_v}
K'\otimes_K K_v\cong K'_{v'}\coloneqq K_v\big(\sqrt{a}\big)\quad\ \text{and}\quad F\otimes_K K_v\cong K_v\big(\sqrt[4]{a}\big).
\end{equation}

Let $T$ be the two-dimensional $K$-torus splitting over $F$
with cocharacter group $M=\Z[\ii]$
on which the generator $1+4\Z\in\Gamma=\Z/4\Z=\Gal(F/K)$ acts by multiplication by $\ii$.
Let $\xi_v\in \Hon(K_v,T)$ be the only non-unit element.
Then $\xi_v^2=1$.
\end{cons}

\begin{lem}\label{l:Kottwitz-1}
There exists an element $\xi\in\Hon(K,T)$ such that $\loc_v(\xi)=\xi_v$ and $\xi^2=1$.
\end{lem}

\begin{proof}
Consider the localization map
\[\loc\colon \Hon(K,T)\to \bigoplus_{w\in\V(K)} \Hon(K_w,T).\]
Since the effective Galois group $\G=\G(F/K)$ is cyclic,
we know that $\ker(\loc)=\Sha(K,T)=1$;
see \cite[Corollary 7.4.4(1)]{BK}.
Thus we can identify $\Hon(K,T)$ with its image $\im(\loc)$. We describe this image.
For any $w\in\V_K$ choose a place $\bw$ of $F$ over $w$.
Then there is a canonical Tate-Nakayama isomorphism
\[\lambda_w\colon \Hon(K_w,T)\isoto \Ho^{-1}(\G_{\bw}\hs, M)= M_\Gbwt\hs, \]
where $\G_\bw$ denotes the stabilizer of $\bw$ in $\G$;
see  Serre \cite[Theorem IX.8.14]{Serre-LF} or  Tate \cite[Theorem on page 717]{Tate}.
Consider the composite homomorphism
\[ \mu_w\colon \Hon(K_w,T)\labelto{\lambda_w} M_\Gbwt\to M_\Gt\hs,\]
where the homomorphism $M_\Gbwt\to M_\Gt$ is the natural projection.
Moreover, consider the map
\[\mu\colon \bigoplus_{w\in\V(K)} \Hon(K_w,T)\,\to\, M_\Gt\hs,\quad\ (\xi_w)\mapsto\sum_{w\in\V(K)}\mu_w(\xi_w).\]
Then
\begin{equation}\label{e:thm-Tate}
\im(\loc)=\ker(\mu);
\end{equation}
see Tate \cite[Theorem on page 717]{Tate}.

Let $v$ be the fixed place of $K$ chosen below the statement of Theorem \ref{thm.main3-b}.
Let $\G_v\subset\G$ denote a decomposition group of $v$  (defined up to conjugacy).
Since by \eqref{e:F-K_v} the tensor product $ F\otimes_K K_v$ is a field, we have $\G_v=\G$.
We saw in Construction \ref{Step2-local} that  $\Hon(K_{v},T)\cong\Z/2\Z$.
Moreover, for the nontrivial element $\xi_v\in \Hon(K_v,T)$  we have
\[\xi_v^2=1,\quad\ \text{but}\quad\  \Res_{K'_{v'}/K_v}(\xi_v)\neq 1\]
where $v'$ is a place of $K'$ over $v$;
see Lemma \ref{l:Res-K'/K}.
Since $\Gamma$ is cyclic, by the Chebotarev density theorem  there exists
a finite place $u$ of $K$, different from $v$, with decomposition group $\G_{\hm u}=\G=\G_{\hm v}$\hs.
Again as in  Construction \ref{Step2-local}, we have $\Hon(K_{u},T)\cong\Z/2\Z$.
Let $\xi_{u}\in \Hon(K_{u},T)$  be the  element with
\[\lambda_{u}(\xi_{u})=-\lambda_v(\xi_v)\in M_{\G\hm_u,\Tors}=M_\Gt=M_\Gvt\hs.\]
For all places $w$ of $K$ different from $v$ and $u$,
we set $\xi_{w}=1\in \Hon(K_{w}, T)$.
Then
\[\mu\big(\hs(\xi_w)_{w\in\V_K}\big)=
    \sum_{w\in\V_K}\mu_w(\xi_w)=\lambda_v(\xi_v)+\lambda_u(\xi_u)=0,\]
and it follows from \eqref{e:thm-Tate}
that there exists a (unique) element $\xi\in \Hon(K,T)$ with
\[\loc_{w}(\xi)=\xi_{w}\quad\ \text{for all}\ \, w\in\V_K.\]
Then $\loc_v(\xi)=\xi_v$ and clearly we have $\xi^2=1$, as desired.
\end{proof}

We now proceed with the global version of Construction~\ref{Step2-local}.

\begin{cons}\label{Step4}
We change our notation.
For $j=1,2,3$, let $T_j$ be the corresponding 2-dimensional torus of Construction \ref{Step3},
and let $\xi_j\in \Hon(K,T_j)$ be a cohomology class as in Lemma \ref{l:Kottwitz-1}.
Set
\[T=T_1\times T_2\times T_3\hs,\quad\  \xi=(\hs\xi_1,\hs\xi_2,\hs\xi_3\hs)\in \Hon(K,T).\]
Then $\xi \neq 1$ but $\xi^2=1$. In other words, $\per(\xi) = 2$.
\end{cons}

\begin{proof}[Conclusion of the proof of Theorem \ref{thm.main3-b}]
It remains to show that $4\,\big|\, \ind(\xi)$. For the sake of contradiction,
assume the contrary: $\xi$ can be split by a finite separable extension $L/K$ such that $4\nmid [L:K]$.

Localizing at our chosen place $v$ of $K$ with residue characteristic not $2$, we obtain
\[\xi_v=\loc_v(\xi) = (\hs\xi_{1,v}\hs,\,\xi_{2,v}\hs,\,\xi_{3,v}\hs) \in \Hon(K_v, T) \]
where $\xi_{j,v}=\loc_v(\xi_j) \in \Hon(K_v, T_j)$ for $j=1,2,3$.
Now observe that
\begin{equation*}
L\otimes_K K_v=\prod_{w \, | \, v}L_w\hs,  \quad \text{whence} \quad  [L:K]=\sum_{w|v}  \; [L_w:K_v].
\end{equation*}
Here $w$ in the last sum runs over the set of places of $L$ over $v$.
Since $4\nmid [L:K]$, there exists a place $w$ over $v$ such that
\begin{equation} \label{e.L_w:K_v}
4\nmid [L_w:K_v].
\end{equation}
Since $L$ splits $\xi$, the extension $L_w/K_v$ splits $\xi_v$, whence $\ind(\xi_v)\,\hs \big|\,\hs[L_w:K_v]$.
 On the other hand, by Construction \ref{Step3-local} and
Lemma~\ref{lem.4|index}, the index $\ind(\xi_v)$ is divisible by $4$ and thus
$4  \mid  [L_w:K_v]$, contradicting~\eqref{e.L_w:K_v}. This completes the proof of
Theorem~\ref{thm.main3-b}.
\end{proof}

\section{Example: non-existence of a power operation}
\label{app:Gille}

\centerline{\em by Philippe Gille}
\bigskip

Our goal is to examine closely a special case of the Galois
cohomology of the split Chevalley group of type $E_8$ denoted in the sequel by $E_8$.
This is mostly based on Bruhat-Tits theory,
and the relevant groups are constructed in \cite[Th.\ 1]{T}
and also in \cite{G3}.

In the field of complex numbers $\CC$, we consider
the standard roots of unity $e^{2\pi i/n}$
which provide isomorphisms $\ZZ/n\ZZ \simlgr \mu_n$,
$\widehat \ZZ \simlgr \widehat \ZZ(1)$ and
$\QQ/\ZZ \simlgr \QQ/\ZZ(1)$.

We consider the field of iterated formal Laurent series  $K=\CC((x))((y))((z))$.
By iterating \cite[\S 7.1, page 17]{GMS}, it follows that the inductive limit $K_s$
of the field extensions $K=\CC((x^{1/n_x}))((y^{1/n_y}))((z^{1/n_z}))$
is an algebraic closure of $K$; furthermore,
we have $\Gal(K_s/K) = (\widehat \ZZ)^3$.

We work with the following extension
$K'=k'((z))$ of $K$ (which is not algebraic)
where  $k'$ is  the maximal prime-to-$5$ extension of $k=\CC((x))((y))$.
We have $\Gal(k_s/k')= (\ZZ_5)^2$.
Again, the inductive limit $K'_s$
of the extensions $k'((z^{1/n}))$
is an algebraic closure of $K'$; furthermore,
we have  $\Gal(K'_s/K') = \widehat \ZZ \times
 \ZZ_5 \times  \ZZ_5$.
In particular, $K'$ is of cohomological dimension $3$, and more precisely,
of cohomological dimension $3$ at $5$ and of cohomological dimension $1$
at the other primes.
The residue map $\partial :
H^3(K',\QQ/\ZZ(2)) \to H^2(k',\QQ/\ZZ(1))$ is an isomorphism.
Taking into account the choice of roots of unity, we have
an isomorphism $$
H^3(K',\QQ/\ZZ(2)) \simla H^2(k',\QQ/\ZZ)=
H^2\bigl(\Gal(K'_s/K'),\QQ/\ZZ \bigr)
=H^2\bigl(\ZZ_5 \times \ZZ_5,\QQ/\ZZ \bigr)=\ZZ/5\ZZ.
$$

\begin{proposition} \label{p:Gille}\

\begin{enumerate}
\item[\rm (1)] The Rost invariant  (as defined in  \cite[Thm.\ 9.10]{GMS})
 $$H^1(K',E_8) \to H^3(K',\QQ/\ZZ(2)) \cong \QQ_5/\ZZ_5$$ is injective and provides a bijection $H^1(K',E_8) \simlgr \ZZ/5\ZZ$.

\smallskip

\item[\rm (2)] There exists a maximal $K$--torus $T$ of $E_8$ such that the composite
$$
H^1(K,T) \to H^1(K,E_8) \xrightarrow{\Rost} H^3(K,\QQ/\ZZ(2))
$$
is not a group homomorphism.

\smallskip

\item[\rm (3)] For the previous $K$-torus $T$, there exists a $1$-cocycle $\xi \in H^1(K,T)$
such that the associated map  $H^1(K',T) \to H^1(K',\,{_\xi E_8})$
has the following property: there exists $\gamma \in
\ker\bigl( H^1(K',T) \to H^1(K',{_\xi E_8}) \bigr)$ such that
$\gamma^2 \not \in \ker\bigl( H^1(K',T) \to H^1(K',{_\xi E_8}) \bigr)$.
\end{enumerate}
\end{proposition}

In (1), the triviality of the kernel
is related to  a result of Chernousov \cite{Che94}.

From this proposition, it follows that there is no functorial in $G$
power operation on $H^1(K',G)$ as in Question \ref{q:main};
see the deduction of Theorem \ref{t:non-ex} from Proposition \ref{p:Gille} at the end of Section \ref{s:number}.

\begin{proof}
It is enough to deal with the mod 5 part
of the Rost invariant; it is of 5-torsion
\cite[Th.\ 16.8 page 150]{GMS}.

\smallskip

\noindent (1)  We consider  the extended Dynkin diagram $\widetilde E_8$
of the Dynkin diagram $E_8$
$$
\begin{picture}(105,45)
\put(155,02){\circle*{2}}
\put(-40,02){\line(1,0){140}}
\put(60,02){\line(0,1){20}}
\put(00,02){\circle*{3}}
\put(20,02){\circle*{3}}
\put(40,02){\circle*{3}}
\put(60,02){\circle*{3}}
\put(80,02){\circle*{3}}
\put(100,02){\circle*{3}}
\put(60,22){\circle*{3}}
\put(20,02){\circle*{3}}
\put(60,02){\circle*{3}}
\put(80,02){\circle*{3}}
\put(100,02){\circle*{3}}
\put(-20,02){\circle*{3}}
\put(-40,02){\circle*{3}}
\put(-48,-11){-$\alpha_0$}
\put(-25,-11){$\alpha_8$}
\put(-5,-11){$\alpha_7$}
\put(15,-11){$\alpha_6$}
\put(35,-11){$\alpha_5$}
\put(55,-11){$\alpha_4$}
\put(75,-11){$\alpha_3$}
\put(95,-11){$\alpha_1$}
\put(55,30){$\alpha_2$}
\end{picture}
$$
\vskip2mm

\noindent We write the Bruhat-Tits decomposition \cite[cor.\ 3.15]{BT3}.
\begin{equation}\label{eq_dec}
\coprod\limits_{ I \subsetneq \widetilde E_8}
H^1(k', M_I)_{\irr} \simlgr H^1(K',E_8)
\end{equation}
where each $M_I$ is  a  split reductive $k'$--group
whose Dynkin diagram is $I$.
The subset $H^1(k', M_I)_{\irr} $ stands for
 irreducible Galois cohomology classes, that is,
not admitting a reduction to any proper parabolic subgroup.
We proceed by case-by-case analysis.
\smallskip

\noindent{\it Case $I=\emptyset$.} $M_I$ is a split $k'$-torus,
so that $1=H^1(k', M_I)_{\irr}$ according to Hilbert 90's theorem.
If follows that $1=H^1(k', M_I)_{\irr}$ which maps
to the distinguished element of $H^1(K',E_8)$.
\smallskip

\noindent{\it Case $I=\widetilde E_8  \setminus \{- \alpha_0\}$.}
We have $M_I=E_8$\hs,
so that $H^1(k',E_8)=1$ in view of a known case
of Serre's conjecture II \cite[lemme 9.3.2]{G2}.

\noindent{\it Case $I=\widetilde E_8  \setminus \{- \alpha_5\}$.}
We have
$M_I= (\SL_5 \times \SL_5)/ \mu_5$ with the embedding
$\mu_5 \to \mu_5 \times \mu_5$, $x \mapsto (x,x^2)$
(using Tits' algorithm \cite[\S 1.7]{T}).
We denote this $k'$--group by $H$.
We denote by $\delta: H^1(k',H) \to H^2(k',\mu_5)$
the boundary map associated to the  isogeny
$1 \to \mu_5 \to \SL_5 \times \SL_5 \to H \to 1$
where the first map is the diagonal map.

The point is that $\delta$
is injective since $k'$ has cohomological dimension $2$
(and reduced norms are onto).  Next we have $$
H^2(k',\mu_5)\simla H^2(k',\ZZ/5\ZZ)= H^2\bigl(\Gal(K'_s/K'),\ZZ/5\ZZ \bigr)
=H^2\bigl(\ZZ_5 \times \ZZ_5,\ZZ/5\ZZ\bigr)=\ZZ/5\ZZ
$$
 and claim that the injective map
  $\delta:  H^1(k',H) \to  H^2(k',\mu_5) \cong \ZZ/5\ZZ$
 is onto. According to a result of Rost--Springer  \cite[Th.\ 1.1]{G3},
 $\mathrm{Im}(\delta)$ contains the image of the boundary map
  $$\delta^\sharp: H^1(k',\PGL_5) \to  H^2(k',\mu_5)\cong \ZZ/5\ZZ$$
  associated to the exact sequence
  $$1 \to \mu_5 \to \SL_5 \to \PGL_5 \to 1.$$
  We use  the commutative diagram, where $k=\CC((x))((y))$,
  \[
\xymatrix{
 \delta^\sharp_{k} :& H^1(k,\PGL_5) \ar[r]\ar[d]^{\Res}&
    H^2(k,\mu_5)\cong \ZZ/5\ZZ \ar[d]^{\Res} \\
 \delta^\sharp_{k'}: &H^1(k',\PGL_5) \ar[r]&   H^2(k',\mu_5)\cong \ZZ/5\ZZ
 }
\]
and since $k'$ is a prime-to-$5$ closure of $k$,
the right-hand-side restriction is injective.
Finally, we know that the map  $\delta^\sharp_{k}$ is onto
by using the  cyclic division
algebras $X^5=x^i$, $Y^5=y$, $YX= e^{2\pi i/5} XY$
for $i=1,2,3,4$. Thus $ \delta^\sharp_{k'}$ is bijective.
We conclude  that $H^1(k',H)$ has five elements,
the trivial class and four irreducible classes.

\smallskip

\noindent{\it Other indices.}
The factors of the $k'$--group $M_{I,\ad}$ occur in the following list:
$$
A_1, A_2, A_3, A_5, A_6, A_7,
D_4, D_5,D_6,D_7,D_8,E_6, E_7 .
$$
The torsion primes of $M_{I,\ad}$ are then included
in $\{ 2,3,7\} $ \cite[\S 5.1]{G2}. Since $\mathrm{cd}_2(k')=
\mathrm{cd}_3(k')=0$, Serre's conjecture I (Steinberg's theorem refined
in \cite[Th.\ 5.2.5]{G2}) states that $H^1(k',M_{I,\ad})=1$
whence $H^1(k', M_I)_{\irr}=\emptyset$.

This case-by-case analysis shows that the decomposition
\eqref{eq_dec} simplifies then
 in a  bijection
\begin{equation*}
H^1(k',H) \simlgr H^1(K',E_8).
\end{equation*}
We use now the compatibility \cite[\S III.3, th.\ 2]{G1}
\[
\xymatrix{
&& H^1(k',H) \ar[d]^{\delta} _{\wr}
\ar[r] &  H^1(K',E_8)  \ar[d]^{r_5} \\
 \ZZ/5\ZZ &=& H^2(k',\mu_5)  & \ar[l]^{ \sim}_{\partial} H^3(K', \mu_5^{\otimes 2}).
 }
\]
Note that the residue map $\partial$ is bijective since
$k'$ has cohomological dimension $2$.
Overall we conclude that  the Rost invariant
is bijective.
\smallskip

\noindent (2) The idea is to use the subgroup $A=\mu_5 \times \mu_5 \times \ZZ/5\ZZ$ of $E_8$
described in \cite[\S 6]{GQM}.
Our choice of a primitive   $5$--root of unity $\omega=e^{2\pi i/5}$
gives rise to   isomorphisms $\ZZ/5\ZZ \simlgr \mu_5$, $A \simlgr (\mu_5)^{3}$ and $\mu_5^{\otimes 2} \cong
\mu_5^{\otimes 3}$.
This reference shows that the composite
$$
 \rho: K^\times /(K^\times)^5 \times K^\times/(K^\times)^5
 \times K^\times/(K^\times)^5= H^1(K,A) \to
 H^1(K,E_8) \xrightarrow{r_5} H^3(K,\mu_5^{\otimes 2})
$$
is $\rho(a,b, c)=  - (a) \cup (b) \cup (c)$.
We have $\rho(a^2,b^2, c^2)=
-  (a^2) \cup (b^2) \cup (c^2) = 3 \, \rho(a,b, c)$,
so that we see that $\rho$ is not a group homomorphism
by testing it on $(x,y,z)$.

We have $\mu_5 \times \mu_5 \subset T_0$ where
$T_0$ is a maximal $K$--split torus of $E_8$
and $\ZZ/5\ZZ$ normalizes $T_0$.

We twist $T_0$ and $E_8$ by the character $\chi$ defined by $z$. It provides a maximal $K$--subtorus $T \subset G$ where
$G$ is a $K$--form of $E_8$. Using the fact that
$H^1(\CC((z)),E_8)=1$, we obtain that $G$ is isomorphic
to $E_8$. We have then a natural
 $K$--subgroup $\mu_5 \times \mu_5 \subset T$. We write  the compatibility
\cite[lemme 7]{G1}
\[
\xymatrix@C=1.1cm{
 \rho: & H^1(K,\mu_5)^3 \ar[d]^{\id- (1,1,z)}
\ar[r] &  H^1(K,E_8) \ar[r]^-{\Rost} \ar[d]^\wr & H^3(K,\mu_5^{\otimes 3}) \cong \ZZ/5\ZZ
\ar[d]^{\id} \\
 \rho_1: 	& H^1(K,\mu_5)^3   \ar[r] &  H^1(K,G) \ar[r]^-{\Rost} & H^3(K,\mu_5^{\otimes 3}) \cong \ZZ/5\ZZ.
 }
\]
It follows that  $\rho_1(x^i,y^j,z^l)=\rho(x^i,y^j,z^{l+1})=
-i\,j \, (l+1)$. In particular
 $$
\rho_1(x^2,y^2,1) = \rho(x^2,y^2,z^{1})=
4 \, \rho(x,y,z^{1})= 4  \, \rho_1(x,y,1) \not = 0.
$$
Since $\mu_5 \times \mu_5 \subset T$,
we conclude that the composite map
\[H^1(K,T) \to H^3(K,G) \xrightarrow{\Rost} H^3(K,\QQ/\ZZ(2))\]
is not a group homomorphism.
\smallskip

\noindent (3) We work again with
$K$ and denote by
$\xi$ a $1$--cocycle representing the image of $(x,y) \in H^1(K,\mu_5)^2$
in $H^1(K,T)$.
From the above diagram we have $\rho_1(x,y,1)=\rho(x,y,z)=-1$, and
the same compatibility provides the commutative
diagram \[
\xymatrix@C=1.1cm{
 \rho_1: & H^1(K,\mu_5)^3 \ar[d]^{\id - (x,y,1)}_{\wr}
\ar[r] &  H^1(K,G) \ar[r]^-{\Rost} \ar[d]^\wr & H^3(K,\mu_5^{\otimes 3}) \cong \ZZ/5\ZZ
\ar[d]^{\id +1}_{\wr} \\
 \rho_2: 	& H^1(K,\mu_5)^3   \ar[r] &  H^1(K,{_\xi G}) \ar[r]^-{\Rost} & H^3(K,\mu_5^{\otimes 3}) \cong \ZZ/5\ZZ.
 }
\]
It follows that
$$
\rho_2(x^i,y^j,z^l)=
\rho_1( x^{i+1}, y^{j+1},z^{l})+1=
 - (i+1) (j+1) (l+1) +1 .
 $$
We have then $\rho_2(x^2,y,1)= -3 \times 2 \times 1 +1= 0$
and $\rho_2(x^4,y^2,1)= -5 \times 3 \times 1 + 1=1 $.
The computation survives when going to $K'$.
Since both Rost invariants over $K'$
 is injective by (1), the class $\gamma$ which is the image of
$(x^2,y)$ in $H^1(K',T)$  belongs to
$\ker\bigl( H^1(K',T) \to H^1(K',{_\xi G}) \bigr)$
but   $\gamma^2 \not \in
\ker\bigl( H^1(K',T) \to H^1(K',{_\xi G}) \bigr)$.
\end{proof}

\begin{remark}{
There are two situations similar to $E_8$ at the prime $5$ \cite[\S 6]{GQM}:
 $G_2$ at  the prime $2$ and $F_4$ at the prime $3$.
For  $G_2$ at $2$, only part (1) of  Proposition \ref{p:Gille} holds (and follows actually from the general result that the Rost invariant is injective for $G_2$ \cite[Cor.\ 33.25]{involutions}).
For $F_4$ at the prime $3$, the whole statement
of the proposition holds.
The $k'$-group $H$ is then $(\SL_3 \times \SL_3)/\mu_3$ attached to the vertex $\alpha_2$ in the extended Dynkin diagram
$$
\begin{array}{ll}
\mbox{$\widetilde F_4 \qquad \qquad $}
\quad
\begin{picture}(100,10)
\put(-20,00){\line(1,0){20}}
\put(00,00){\line(1,0){20}}
\put(20,1.1){\line(1,0){30}}
\put(20,-1.2){\line(1,0){30}}
\put(30,-2.5){$>$}
\put(50,00){\line(1,0){20}}
\put(-20,0){\circle*{3}}
\put(00,0){\circle*{3}}
\put(20,0){\circle*{3}}
\put(50,0){\circle*{3}}
\put(70,0){\circle*{3}}
\put(-35,-11){$-\alpha_0$}
\put(-5,-11){$\alpha_1$}
\put(15,-11){$\alpha_2$}
\put(45,-11){$\alpha_3$}
\put(65,-11){$\alpha_4$}
\end{picture}
\end{array}
$$
\vskip4mm
\noindent We use then the  subgroup $\mu_3 \times \mu_3 \times \ZZ/3\ZZ$.
At the end of the proof, we
take the class $\gamma \in H^1(K', T)$ corresponding to $(x,y,1)$}.
\end{remark}

\end{document}